\documentclass[a4paper,12pt]{article}
\usepackage{amsmath,amsthm,amssymb,amsfonts}
\usepackage{bbm,bm}
\usepackage{latexsym}
\usepackage{mathrsfs}
\usepackage{threeparttable}
\usepackage{tabularx}
\usepackage{booktabs}
\usepackage{graphicx,subfigure}
\usepackage{color}
\usepackage{indentfirst}
\usepackage{geometry}
\usepackage{caption}
\usepackage{lineno}
\usepackage{abstract}
\usepackage{enumerate,mdwlist}
\usepackage{adjustbox}
\usepackage[numbers,sort&compress]{natbib}
\usepackage[colorlinks,
            linkcolor=blue, 
            citecolor=red  
            ]{hyperref}
\usepackage{epstopdf}

\graphicspath{{figures/}}

\def \[{\begin{equation}}
\def \]{\end{equation}}

\newtheorem{thm}{Theorem}[section]

\newtheorem{observation}{Observation}
\newtheorem{lem}[thm]{Lemma}
\newtheorem{cor}[thm]{Corollary}

\newenvironment{kst}

\begin{document}

\title{2-extendability of (4,5,6)-fullerenes\footnote{This work is supported by 
the National Natural Science Foundation of China (Grant Nos. 11871256, 12271229, 12101220) and the Foundation of Hunan Provincial Education Department (Grant No. 21C0445).}}
\author{ Lifang Zhao$^{a,b}$, Heping Zhang${^b}$\footnote{Corresponding author.}\\
{\small $^a$College of Science, Hunan University of Technology, Zhuzhou, Hunan 412000, P.R. China}\\
{\small $^b$School of Mathematics and Statistics, Lanzhou  University,  Lanzhou, Gansu 730000, P.R. China}\\
{\small E-mails: zhaolifang@hut.edu.cn, zhanghp@lzu.edu.cn}}
\vspace{0.5mm}

\date{}

\maketitle

\noindent {\bf Abstract}: A (4,5,6)-fullerene is a plane cubic graph whose faces are only quadrilaterals, pentagons and hexagons, which  includes all (4,6)- and (5,6)-fullerenes.  A connected graph $G$ with at least $2k+2$ vertices  is $k$-extendable if $G$ has  perfect matchings  and  any matching of size $k$ is contained in a perfect matching of $G$.
We know that each (4,5,6)-fullerene graph is 1-extendable and at most 2-extendable.
It is natural to wonder which (4,5,6)-fullerene graphs are 2-extendable. In this paper, we completely solve this problem (see Theorem \ref{th2}): All non-2-extendable (4,5,6)-fullerenes consist of four sporadic (4,5,6)-fullerenes ($F_{12},F_{14},F_{18}$ and $F_{20}$) and five classes of (4,5,6)-fullerenes.
As a surprising consequence, we find that all (4,5,6)-fullerenes with the anti-Kekul\'{e} number 3 are non-2-extendable. Further, there also always exists a non-2-extendable (4,5,6)-fullerene with arbitrarily even $n\geqslant10$ vertices.

\vspace{2mm} \noindent{\it Keywords}: (4,5,6)-fullerene; perfect matching; 2-extendability;
anti-Kekul\'{e} number

\vspace{2mm}


{\setcounter{section}{0}
\section{\normalsize Introduction}\setcounter{equation}{0}
An edge subset $M$ of a graph $G$ is called a \textit{matching} if any two edges of it have no an endvertex in common.  A \textit{perfect matching} (or Kekul\'{e} structure in chemical literature) of a graph $G$ is a matching such that every vertex of $G$ is incident with one edge of it.
A connected graph $G$ with at least $2n+2$ vertices is said to be \textit{$n$-extendable} if $G$ has a matching of size $n$ and each such matching is contained in a perfect matching \cite{4}, while the maximum integer $k$ such that $G$ is $k$-extendable is called the \textit{extendability} of $G$.
As we known,  no planar graph is 3-extendable \cite{11}. Further, 
 any 2-edge connected cubic graph is 1-extendable by a result of Berge \cite{12} 
  (see also  \cite{13, 4}).
Hence, an interesting problem is to characterize  the 2-extendable planar graphs.
In 1987, Holton and Plummer \cite{1} showed that a 3-regular, 3-connected planar graph which is cyclically 4-edge connected and has no faces of size 4 is 2-extendable. They \cite{7} also proved that an $(n+1)$-regular, $(n+1)$-connected bipartite graph with cyclical edge-connectivity at least $n^2$ is $n$-extendable.
Such two results implies that all (5,6)-fullerenes
are 2-extendable \cite{2}, and all (4,6)-fullerenes
with cyclical edge-connectivity 4 are also 2-extendable \cite{3} respectively. But, (4,6)-fullerenes with cyclical edge-connectivity 3 (denoted by $\mathcal{T}$, which is denfined in Section 2) are non-2-extendable\cite{3}.
Moreover, Plummer \cite{14} showed that all 5-connected even planar graphs are 2-extendable.
For a bipartite graph,
Lakhal et al. \cite{Al-98} and Zhang et al. \cite{Al-06} showed that there is a polynomial algorithm to determine its extendability. However, for general graphs, Hackfeld and Koster \cite{complexity of GN}  presented  that the extendability problem is co-NP-complete.
For more works on matching extension, one can see refs. \cite{D.Lou,Q.Yu1,18,17}.

A (4,5,6)-fullerene (graph) is a plane (or spherical) cubic graph whose faces are only quadrilaterals, pentagons and hexagons. 
(5,6)-fullerenes and (4,6)-fullerenes are special (4,5,6)-fullerenes that does not contain quadrilaterals and pentagons respectively.
Further, (5,6)-fullerenes have the cyclical edge-connectivity 5 and (4,6)-fullerenes have the cyclical edge-connectivity 3 or 4 \cite{c-(56), c-(46)}. No matter from the view of mathematic or  chemistry, the structural properties of (4,6)- and (5,6)-fullerenes have been extensively investigated. For mathematical studies of fullerenes, one can see the recent survey \cite{(56)-survey} and references contained in it.
We know that each (4,5,6)-fullerene is
1-extendable. Hence, we  investigate the 2-extendability of (4,5,6)-fullerenes in this paper.

In the next section, we recall some concepts and results needed in our discussions.  In Section 3, by using the strength Tutte's theorem on perfect matching of graphs, we characterize all (4,5,6)-fullerenes that are non-2-extendable, which consist of $F_{12}, F_{14}, F_{18}, F_{20}$, all  (4,5,6)-fullerenes in $\mathcal{T}$ and the other four classes of (4,5,6)-fullerenes having at least 2 and at most 10 pentagons.
Specially, we find that all (4,5,6)-fullerenes with anti-Kekul\'{e} number 3 are non-2-extendable.
 Further, for any even $n\geqslant10$,  there always exists a (4,5,6)-fullerene with $n$ vertices being non-2-extendable.

\section{\normalsize Definitions and preliminaries}\setcounter{equation}{0}

Throughout this paper, we only consider simple graphs. For a graph $G$,  let $V(G)$ denote the vertex-set, $E(G)$ the edge-set and $F(G)$ the face set if $G$ is a plane graph. For the  notation and terminologies not stated in this paper, a reader is referred to \cite{5,4}.

For a graph $G$, denote $N_G(v)$ or briefly $N(v)$ the set of \textit{neighbors} of a vertex $v\in V(G)$.
The number of edges incident with $v$ in $G$ is called the \textit{degree} of $v$, written by $d_G(v)$ or simply by $d(v)$.
$S\subseteq V(G)$ is called an \textit{vertex independent set} if any two vertices of $S$ is non-adjacent.
For $S\subseteq V(G)$, $G-S$  is  the subgraph  obtained from $G$ by deleting all the vertices in $S$ and their incident edges,  while  $G[S]$ is the \textit{induced subgraph}  by $S$ defined as the subgraph $G-\bar{S}$, where $\bar{S}=V(G)-S$. For two subgraphs $H_1$ and $H_2$ of $G$, denote $H_1\cap H_2$ be the graph with vertex set $V(H_1)\cap V(H_2)$ and edge set $E(H_1)\cap E(H_2)$.
For a graph $G$, $S\subseteq V(G)$ is said to be a \textit{vertex cut} of $G$ if $G-S$ is disconnected, and
$G$ is said to be \textit{$k$-connected} if $k< |V(G)|$ and $G-S$ is connected for any $S\subseteq V(G)$ with $|S|<k$.
The \textit{connectivity} of $G$ is  the greatest integer $k$ for which $G$ is $k$-connected.
For $E_0\subseteq E(G)$, $G-E_0$ is the subgraph of $G$ obtained by deleting the edges of $E_0$ and $V(E_0)$ is the set of the endvertices of edges in $E_0$.
For $S\subseteq V(G)$, let $\overline{S}=V(G)\setminus S$. Let $[S, \overline{S}]$ be the set of edges of $G$ with one endvertex in $S$ and the other one in $\overline{S}$. We call $[S, \overline{S}]$  an \textit{edge cut} if both $S$ and $\overline{S}$ are not empty. Particularly, $[S, \overline{S}]$ is called \textit{trivial} if $S$ or $\overline{S}$  only contains one vertex. A connected graph $G$ is said to be \textit{cyclically $k$-edge connected} if $G$ cannot be separated into two components, each containing a cycle, by deleting less than $k$ edges. The maximum $k$ such that $G$ is cyclically $k$-edge connected is called the \textit{cyclical edge-connectivity} of $G$.
 A cycle of a plane graph is called a \textit{facial cycle} if it is the boundary of a face. Denote by $P_n$  a path with $n$ vertices.

Let $T_n$ be a (4,5,6)-fullerene that consists of $n$ concentric layers of hexagons and capped on each end by a cap formed by three quadrilaterals with one common vertex. Further, we call the edges between two concentric 6-cycles (each of which is not the boundary of a face) \textit{traversed edges}, such as $T_3$; see Fig. \ref{T_n}. Denoted by $\mathcal{T}=\{T_n|n\geqslant1\}$.
\begin{figure}[!htbp]
\begin{center}
\includegraphics[totalheight=26mm]{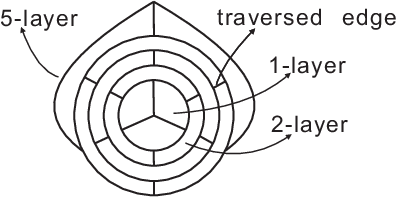}
 \caption{\label{T_n}\small{Illustration for a (4,5,6)-fullerene $T_3$ in $\mathcal{T}$, layer and traversed edge.}}
\end{center}
\end{figure}

\begin{lem}[\cite{8}]\label{huan}
A (4,5,6)-fullerene graph is cyclically 4-edge connected if and only if it does not belong to  $\mathcal{T}$.
\end{lem}

\begin{lem}[\cite{8}]\label{lem1}
Every (4,5,6)-fullerene has connectivity 3 and
 every 3-edge-cut of a (4,5,6)-fullerene not contained in $\mathcal{T}$ is trivial
\end{lem}

\begin{lem}[\cite{8}]\label{lem2}
Let $F$ be a (4,5,6)-fullerene. Then $F$ has no cycles with length 3 and every cycle with length 4 or 5 of $F$ is a facial cycle.
\end{lem}

For the perfect matchings of $T_n\in \mathcal{T}$, Jiang et al. gave the following property.

\begin{lem}[\cite{3}]\label{T_n pm}
  Let $M$ be any perfect matching of $T_n$. Then $M$ contains exactly one traversed edge in each layer. Conversely, any set containing exactly one traversed edge from each layer is extended to a unique perfect matching.
\end{lem}

Clearly,  there is no perfect matching of $T_n$ containing any two traversed edges of $i$-layer ($2\leqslant i\leqslant n+1$) by Lemma \ref{T_n pm}. Hence, we have the following result.
\begin{lem}\label{2-extendable of T_n}
Each (4,5,6)-fullerene in $\mathcal{T}$ is non-2-extendable.
\end{lem}

A graph $G$ is called \textit{factor-critical} if  $G-v$ has a perfect matching for each vertex $v\in V(G)$.
Denote the components of $G-S$ by $\mathcal{C}_{G-S}$, where $S\subseteq V(G)$. Let $G_S$ be the graph which arises from $G$ by contracting every component of $\mathcal{C}_{G-S}$ to a single vertex and deleting all the edges inside $S$. If $G_S$ contains a matching $M$ such that every vertex of $S$ is incident with an edge in $M$, then we call the vertex set $S$ \textit{matchable} to $\mathcal{C}_{G-S}$.
 Thus, we will use the following general result.

\begin{thm}[\cite{5}]\label{th1}
  Every graph $G$ contains a vertex subset $S$ with the following two properties:
  \begin{kst}
\item[{\rm (i)}]$S$ is matchable to $\mathcal{C}_{G-S}$;
\item[{\rm (ii)}] Every component of $G-S$ is factor-critical.
\end{kst}
  Given any such set $S$, the graph $G$ contains a perfect matching if and only if $|S|=|\mathcal{C}_{G-S}|$.
\end{thm}

\section{\normalsize Main results}\setcounter{equation}{0}

As we known,  all (4,6)-fullerenes with cyclical edge-connectivity 4 and all (5,6)-fullerenes are 2-extendable \cite{2,3}.
By Lemma \ref{2-extendable of T_n}, we know $F\in\mathcal{T}$ is non-2-extendable. But for a (4,5,6)-fullerene which contains not only quadrilaterals but also pentagons, it may be non-2-extendable.
For instance, (4,5,6)-fullerenes $F_{12}, F_{14}, F_{18}$ and $F_{20}$ shown in Fig. \ref{F_{12}} each  has no perfect matching containing  the two bold edges.
\begin{figure}[!htbp]
\begin{center}
\includegraphics[totalheight=27mm]{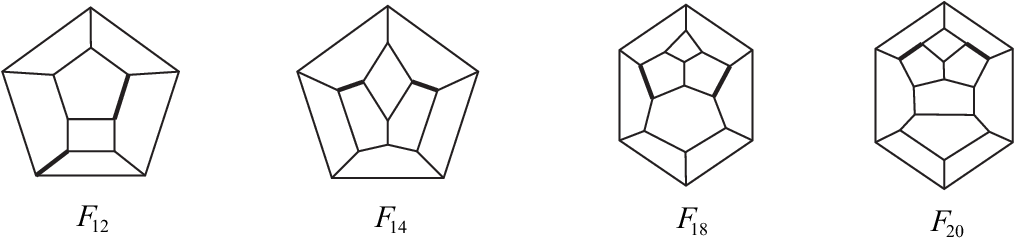}
 \caption{\label{F_{12}}\small{(4,5,6)-fullerenes $F_{12}, F_{14}, F_{18}$ and $F_{20}$  are non-2-extendable.}}
\end{center}
\end{figure}

We characterize all 2-extendable (4,5,6)-fullerene graphs by using the method of  forbidden subgraphs. Hence, we firstly define 29 forbidden subgraphs $H_i$($1\leqslant i\leqslant 29$) and 93 classes of forbidden subgraphs
$H_i$($30\leqslant i\leqslant 122$)    as in Figs. \ref{subs} and \ref{subs1}. For $30\leqslant i\leqslant 122$,  we simply use $H_i$ instead of $H_i^n$ if we do not specially point out how many concentric layers of hexagons $H_i$ contains.
For instance, the subgraphs $H_{30}$ contains  three quadrilaterals,  one pentagon and each layer contains exactly five hexagons.
In particular, the former 21 subgraphs have been already defined by the two present authors \cite{9}.
For each $22\leqslant i\leqslant122$, we can see that there is at least one quadrilateral and one pentagon in $H_i$.
Let $W$ and $B$ be the sets of white vertices and black vertices in $H_i$($22\leqslant i\leqslant122$) respectively.
Then, we have $|W|=|B|+3$, i.e., $V(H_i)$ is odd. Further, $B$ is a vertex independent set of $H_i$ and every vertex of $B$ is of degree 3 in $H_i$.

Let $F$ be a (4,5,6)-fullerene. Denote $p_4$ and $p_5$ the number of quadrilaterals and pentagons of $F$ respectively. Then, $2p_4+p_5=12$ \cite{8}.  Let
$$
\begin{array}{llll}
  \mathscr{F}_1=\{F\mid p_5=4 \text{ and } H_1,  H_3,  H_4,  H_{17},  H_{18},  H_{19} \text{ or }  H_{20}\subseteq F  \},  \\[5pt]
  \mathscr{F}_2=\{F\mid p_5=2 \text{ and }H_2 \text{ or }  H_{21} \subseteq F \}, \\[5pt]
  \mathscr{F}_3=\{F\mid  p_5=6 \text{ and } H_i \subseteq F \text{ for some } 5\leqslant i\leqslant 16\}, \\[5pt]
  \mathscr{F}_4=\{F\mid  H_i \subseteq F \text{ for some } 22\leqslant i\leqslant 122 \}.
\end{array}
$$
 Meanwhile, $2\leqslant p_5\leqslant10$ for $F\in \mathscr{F}_4$. Clearly, all pentagons of $F$ are only contained in $H_k$ if $F\in \cup_{i=1}^3\mathscr{F}_i$ and $H_k\subseteq F$, where $1\leqslant k\leqslant21$.

An \textit{anti-Kekul\'{e} set}  of a connected graph $G$ is an edge subset whose deletion leads to the subgraph being connected but having no perfect matchings.
The smallest cardinality of  anti-Kekul\'{e} sets of $G$ is called the \textit{anti-Kekul\'{e} number} of $G$, denoted by $ak(G)$,which was introduced by Vuki\v{c}evi\'{c} and Trinajsti\'{c} \cite{def anti-ke}.
The two present authors gave the following result.
\begin{thm}[\cite{9}]\label{2020-ak}
Any (4,5,6)-fullerene $F$ has anti-Kekul\'{e} number 3 or 4. Further, $ak(F)=3$ if and only if $F\in\cup^3_{i=1}\mathscr{F}_i\cup \{F_{12},F_{14},F_{18}, F_{20}\}$.
\end{thm}

\begin{figure}[!htbp]
\begin{center}
\includegraphics[totalheight=220mm]{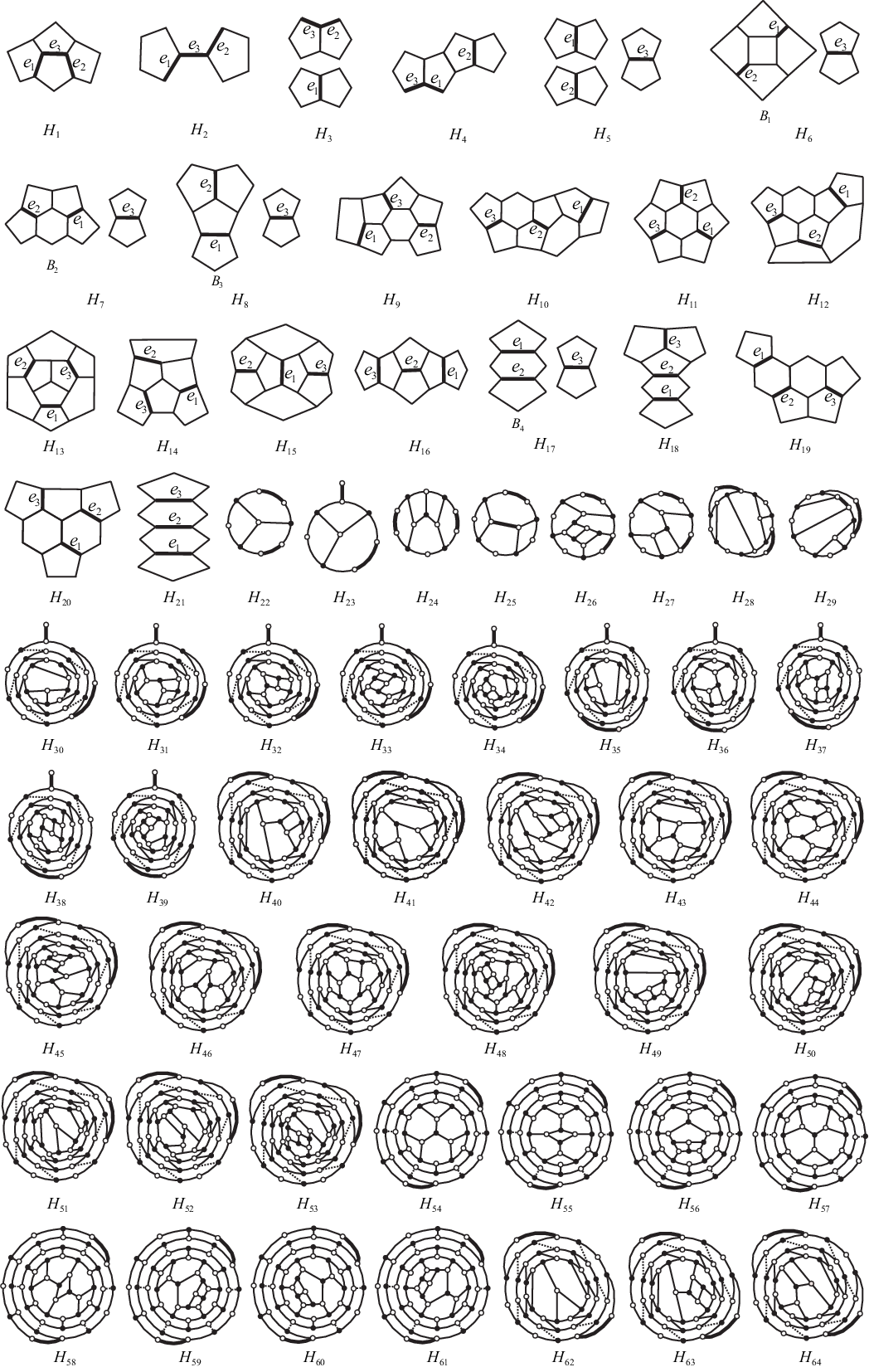}
 \caption{\label{subs}\small{The possible subgraphs $H_i$($1\leqslant i\leqslant 64$). 
 }}
\end{center}
\end{figure}

\begin{figure}[!htbp]
\begin{center}
\includegraphics[totalheight=210mm]{
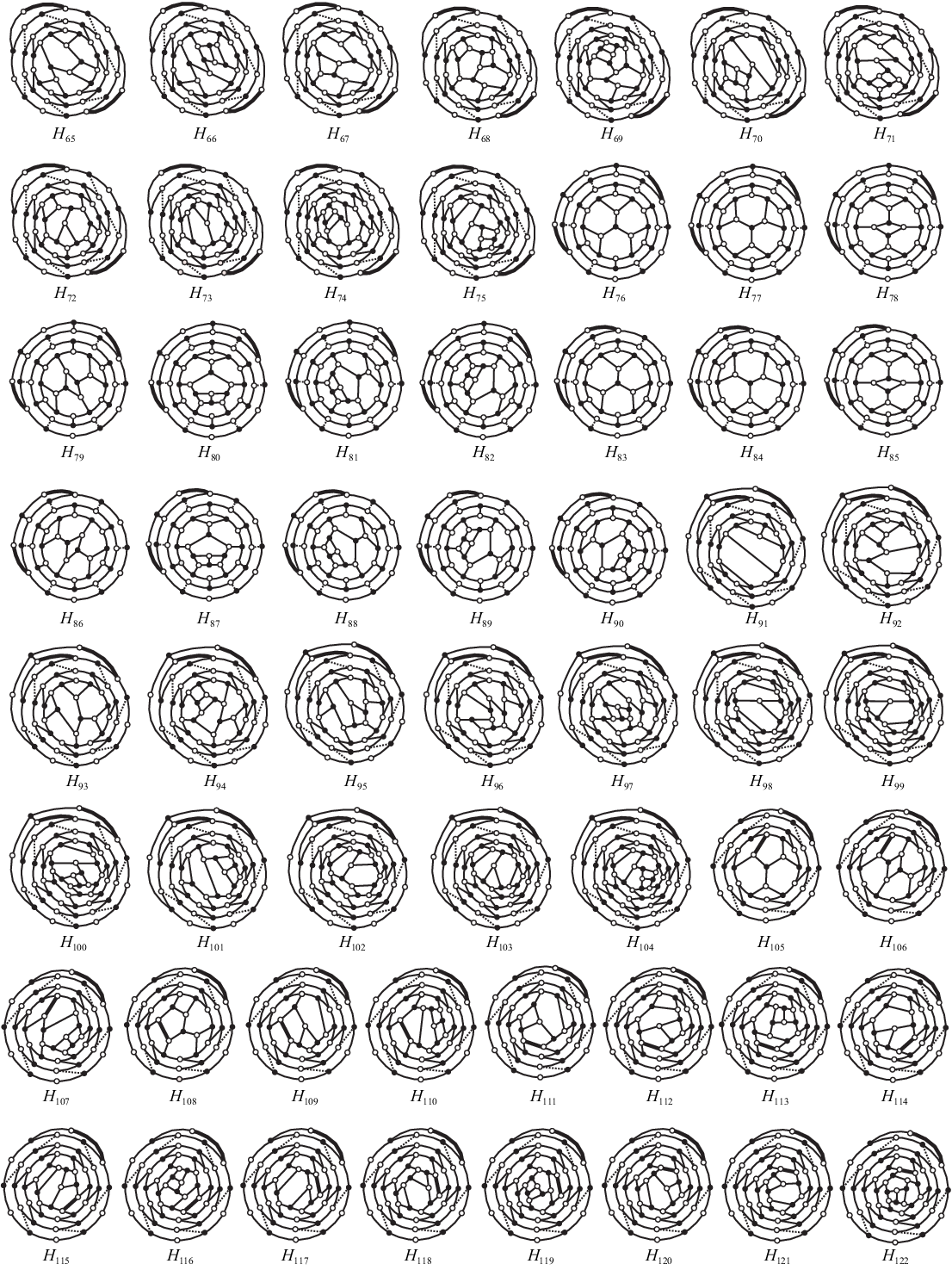}
 \caption{\label{subs1}\small{The possible subgraphs 
 $H_i$, $65\leqslant i\leqslant 122$.}}
\end{center}
\end{figure}

\begin{lem}\label{sufficiency}
If $F\in \cup^4_{i=1}\mathscr{F}_i\cup \mathcal{T}\cup\{F_{12},F_{14},F_{18}, F_{20}\}$, then $F$ is non-2-extendable.
\end{lem}
\begin{proof}
From the beginning of this section and by Theorem \ref{2-extendable of T_n},
we know that $F$ is non-2-extendable for $F\in \mathcal{T}\cup\{F_{12},F_{14},F_{18}, F_{20}\}$.
For $F\in \cup^3_{i=1}\mathscr{F}_i$, $H_k\subseteq F$ for some $1\leqslant k\leqslant21$. Let $E_k$ be the set of three bold edges $e_1,e_2,e_3$ of $H_k$($1\leqslant k\leqslant21$) in Fig. \ref{subs}. Let $F_k'=F-E_k$.
From the proof of Lemma 3.2 in Ref \cite{9},
we have that $F_k'$ is a connected bipartite graph. Let $(X, Y)$ be a bipartition of $F_k'$, and every edge of $E_k$ has its endvertices in the same partite set of $F_k'$. We claim that $V(E_k)\subseteq X$ or $Y$.
If not all vertices of $V(E_k)$ belong to the same partite set of $F_k'$, then we have $3|X|-2=3|Y|-4$ or $3|X|-4=3|Y|-2$, i.e., $3||X|-|Y||=2$, a contradiction. The claim holds.
Without loss of generality, we may assume $V(E_k)\subseteq X$. Thus,  $3|X|-6=3|Y|$. 
  But if there is a perfect matching of $F$ containing $\{e_1,e_2\}$, then it must have the equality $|X|-4=|Y|$ or $|X|-6=|Y|$, which contradicts the fact $|X|-2=|Y|$. Hence,  $F\in \cup^3_{i=1}\mathscr{F}_i$ is non-2-extendable.

Let $F\in \mathscr{F}_4$. Then there is a $22\leqslant k\leqslant122$ such that $H_k\subseteq F$.
Thus,
we have $|W|=|B|+3$, where $W$ and $B$ are the sets of white vertices and black vertices  in $H_k$ respectively.
Let $E_k$ be the set of two bold edges in $H_k$. 
Clearly, $V(E_k)\subseteq W$. Denote $F_k'=F-V(E_k)$ and $W'=W\setminus V(E_k)$.
Since
each vertex in $B$ has three neighbours in $W$, $F-W=F_k'-W'$ has at least $|B|=|W'|+1$ isolated vertices as odd components. By Tutte's Theorem, $F_k'$ has no perfect matching, that is, $F$ has no perfect matching containing $E_k$. Hence, each $F\in \mathscr{F}_4$ is non-2-extendable.
\end{proof}

Next, we will prove that the converse of Lemma \ref{sufficiency} is also true. Then, we can get the main result of this article as follows.

 A subgraph $H$ of a (4,5,6)-fullerene $F$ is called a \textit{patch} if its boundary is a cycle with only 2-degree or 3-degree vertices while other vertices (called \textit{internal vertices}) have 3-degree and every inner face of $H$ is also a face of $F$.

\begin{thm}\label{th2}
Let $F$ be a (4,5,6)-fullerene. Then $F$ is non-2-extendable if and only if $F\in \cup^4_{i=1}\mathscr{F}_i\cup \mathcal{T}\cup\{F_{12},F_{14},F_{18}, F_{20}\}$.
\end{thm}
\begin{proof}
 By Lemma \ref{sufficiency}, it suffices to show  the necessity.
Since $F\in \mathcal{T}$ is non-2-extendable by Lemma \ref{2-extendable of T_n}, we consider $F\notin \mathcal{T}$ in the sequel.

Suppose $F$ is non-2-extendable. Then there exists a matching  $E_0=\{e_1, e_2\}$ of size 2 such that $F-V(E_0)$ has no perfect matching.
 Let $e_1=u_1u_2$, $e_2=u_3u_4$ and $F_0=F-V(E_0)$.  Since $F_0$ has no perfect matching, $F_0$ has a subset $X_0$ satisfying the two properties of Theorem \ref{th1} and $|\mathcal{C}_{F_0-X_0}|\geqslant |X_0|+1$.
 Moreover, since $|\mathcal{C}_{F_0-X_0}|$ and $|X_0|$ have the same parity, 
 $|\mathcal{C}_{F_0-X_0}|\geqslant |X_0|+2$. Let $G_i$ be the components of $F_0-X_0$,
$m_i$ the number of edges in $F$ between $G_i$ and $X_0$, $r_0$ the number of edges in $F$ between $V(E_0)$ and $X_0$, and $r_i$ the number of edges in $F$ between $V(E_0)$ and $G_i$, $i=1,2, \ldots, |\mathcal{C}_{F_0-X_0}|$. Then, $m_i+r_i\geqslant3$ by Lemma 
\ref{lem1},  $1\leqslant i\leqslant|\mathcal{C}_{F_0-X_0}|$.
Therefore,
 \begin{equation}
  \label{eq00}
\begin{aligned}
3|\mathcal{C}_{F_0-X_0}|&\leqslant \sum_{i=1}^{|\mathcal{C}_{F_0-X_0}|}(m_i+r_i)=3|X_0|+12-2|E(F[X_0])|-2|E(F[V(E_0)])|-2r_0
\end{aligned}
\end{equation}
 If $|\mathcal{C}_{F_0-X_0}|\geqslant|X_0|+4$, then Ineq. (\ref{eq00}) implies $|E(F[X_0])|+|E(F[V(E_0)])|+r_0\leqslant0$, a contradiction to $|E_0|=2$. Hence, $|\mathcal{C}_{F_0-X_0}|=|X_0|+2$, and Ineq. (\ref{eq00}) implies
 \begin{equation}
  \label{eq1}
\begin{aligned}
0&\leqslant 6-2|E(F[X_0])|-2|E(F[V(E_0)])|-2r_0.
\end{aligned}
\end{equation}
Obviously, $2\leqslant|E(F[X_0])|+|E(F[V(E_0)])|+r_0\leqslant3$.

If $|E(F[X_0])|+|E(F[V(E_0)])|+r_0=3$, then 
$m_i+r_i=3$, i.e., $|V(G_i)|=1$ by Lemma 
 \ref{lem1}, where $1\leqslant i\leqslant|\mathcal{C}_{F_0-X_0}|$. Denote $Y_0$ the set of all singletons $y_i$ from each $G_i$ ($1\leqslant i\leqslant|\mathcal{C}_{F_0-X_0}|$) and $E_0'$ the set of the three edges in $F[V(E_0)\cup X_0]$. Then, $F-E_0'=(X_0\cup V(E_0), Y_0)$ is bipartite with $|X_0\cup V(E_0)|=|X_0|+4>|Y_0|$. So, $F-E_0'$ has no perfect matching and is connected by Lemma 
 \ref{lem1}. Hence, $E_0'$ is an anti-Kekul\'e set of $F$ and we can get $F\in\cup_{i=1}^3\mathscr{F}_i\cup\{F_{12},F_{14},F_{18},F_{20}\}$  by Theorem 3.2 in Ref \cite{9}.

If $|E(F[X_0])|+|E(F[V(E_0)])|+r_0=2$, then  $|E(F[X_0])|=r_0=0$, $|E(F[X_0])|=2$ and the inequality in Ineq. (\ref{eq1}) is strict. Since $|V(G_i)|$ is odd, which means $m_i+r_i$ is odd, there exists an $i_0$ such that $m_{i_0}+r_{i_0}=5$ and $m_i+r_i=3$ for any $i\neq i_0$. Without loss of generality, assume $i_0=1$. Thus, $|V(G_1)|\geqslant3$ and $|V(G_i)|=1$ for $2\leqslant i\leqslant |\mathcal{C}_{F_0-X_0}|$. Let $Y_0$ denote the set of all singletons $y_i$ from each $G_i$, $2\leqslant i\leqslant |\mathcal{C}_{F_0-X_0}|$. Note that the graph $F_0'$ which is obtained from $F-E_0$ by contracting the subgraph $G_1$ to a single vertex $G_1'$ is bipartite while $(X_0\cup V(E_0), Y_0\cup\{G_1'\})$ is a bipartition of $F_0'$.
For the sake of clarity,  we color the vertices 
$X_0\cup V(E_0)$ in white and $Y_0\cup V(G_1)$ in black. Moreover, for the analysis more convenience, we remark the vertices of $A_0=X_0\cup V(E_0)$ by $a_i$, $1\leqslant i\leqslant|X_0|+4$. Thus, $|E(F[A_0])|=|E_0|=2$.

Since $G_1$ is factor-critical, $G_1$ is 2-edge connected. Let $C_1$ be the boundary of the face of $G_1$ but not a face of $F$. Then, $C_1$ is a cycle.
Since $m_1+r_1=5$, $G_1$ is a patch of $F$ by Lemma 
 \ref{lem1} and there are exactly five 2-degree vertices of $G_1$ which lie on the cycle $C_1$. Without loss of generality, let $b_i$  be the five 2-degree vertices on $C_1$ in the clockwise direction and $a_i\in A_0$ the neighbor of $b_i$, $1\leqslant i\leqslant5$.
 By Lemma 
 \ref{lem1}, $F-G_1$ has no isolated vertex. We claim $F-G_1$ has at most one pendent vertex. Otherwise, $F-G_1$ will have two pendent vertices and each pendent vertex of $F-G_1$ is adjacent to two successive vertices of $\{b_1,\ldots, b_5\}$
 by Lemma 
 \ref{lem1}.  Without loss of generality, assume $a_1=a_2$ and $a_3=a_4$, i.e., $a_1(a_2)$ and $a_3(a_4)$ are two pendent vertices of $F-G_1$. Then, by Lemma 
 \ref{lem1}, we have $\{a_1a_5,a_3a_5\}\subseteq E(F)$. That is, $E_0=\{a_1a_5,a_3a_5\}$, a contradiction as $E_0$ is a matching of $F$. Our claim holds. Further, we have the following result:

\begin{observation}\label{claim2}
\emph{ Let $f$ be a face of $F-G_1$ and $F$. Then (i) if $|E(f)\cap E_0|=0$, then $f$ is an even face; (ii) if $|E(f)\cap E_0|=1$, then $f$ is a pentagon; and (iii) if $|E(f)\cap E_0|=2$, then $f$ is a hexagon.}
\end{observation}

For $1\leqslant i\leqslant5$, let $f_i$ be the face of $F$ containing edge set $\{a_ib_i,a_{i+1}b_{i+1}\}$, see Fig. \ref{figf_i}, where $a_6=a_1$ and $b_6=b_1$. We claim any two faces of $f_1,f_2,f_3,f_4$ and $f_5$ are different.
By Lemma
\ref{lem1}, any two successive faces are different. If $f_1=f_3$, then, by the restriction on the faces of $F$, $a_1=a_4$ and $a_2=a_3$, which means $a_1$ and $a_2$ are two pendent vertices of $F-G_1$, a contradiction as $F-G_1$ has at most one pendent vertex. 
Similarly, we can also get $f_i$ and $f_j$ which are apaced by a face are different, $1\leqslant i,j\leqslant5$. Our claim holds.

By the restriction on the faces of $F$, $|E(f_i)\cap E(F-G_1)|\leqslant3$. Further, since $F_0'$ is bipartite,  $|E(f_i)\cap E(F-G_1)|=0$ or 2 if $E(f_i)\cap E_0=\emptyset$, while  $|E(f_i)\cap E(F-G_1)|=1$ or 3 if $E(f_i)\cap E_0\neq\emptyset$.
 Let $|E(f_1)\cap E(F-G_1)|\leqslant |E(f_i)\cap E(F-G_1)|$ for $2\leqslant i\leqslant5$. Thus, $0\leqslant|E(f_1)\cap E(F-G_1)|\leqslant2$ as $|E_0|=2$. Next, we consider the size of $E(f_1)\cap E(F-G_1)$ to obtain the following facts, i.e, $F\in \mathscr{F}_4$, which complete the entire proof.

\textbf{Fact 1.} If $|E(f_1)\cap E(F-G_1)|=0$, then there is a subgraph $H\in \cup_{i=30}^{39}\mathcal{H}_i\cup\{H_{22}, H_{23}, H_{24}\}$ such that $H\subseteq F-G_1$.

\textbf{Fact 2.} If $|E(f_1)\cap E(F-G_1)|=1$, then there is a subgraph $H\in\{H_{22}, H_{23}, H_{25}, H_{26},H_{27}\}$ such that $H\subseteq F-G_1$.

\textbf{Fact 3.} If $|E(f_1)\cap E(F-G_1)|=2$, then there is a subgraph $H\in \cup_{i=40}^{122}\mathcal{H}_i\cup\{H_{22}, H_{25}, H_{28},H_{29}\}$ such that $H\subseteq F-G_1$.

The proofs of such three facts will be given in next three sections.
\end{proof}

Combining with Theorems \ref{2020-ak} and \ref{th2}, the following result obviously holds.

\begin{cor}\label{last}
Every (4,5,6)-fullerene with anti-Kekul\'{e} number 3 is non-2-extendable.
\end{cor}

\begin{cor}[\cite{9}]\label{ak-any}
There is a (4,5,6)-fullerene with $n$ vertices having anti-Kekul\'e number 3 for any even $n\geqslant10$.
\end{cor}

By Theorems  \ref{last} and   \ref{ak-any}, the following result also holds.
\begin{cor}
For any even $n\geqslant 10$, there exists a (4,5,6)-fullerene 
with $n$ vertices which is non-2-extendable.
\end{cor}

Considering that every (4,6)-fullerene has no pentagons while every (5,6)-fullerene has no quadrilaterals,  the following result is obtained immediately by Theorem \ref{th2}.

\begin{cor}[\cite{2,3}] \label{cor1}
All (4,6)-fulerene but not in $\mathcal{T}$ and (5,6)-fullerenes are 2-extendable.
\end{cor}

\section{The Proof of Fact 1 }

Let $|E(f_1)\cap E(F-G_1)|=0$, i.e., $a_1=a_2$.  Then $a_1(a_2), a_3, a_4$ and $a_5$ are pairwise different as $F-G_1$ has no isolated vertex and at most one pendent vertex.
If $E_0\cap E(f_2)=\emptyset$ and $E_0\cap E(f_5)=\emptyset$, then 
$|E(f_2)\cap E(F-G_1)|=2$ and $|E(f_5)\cap E(F-G_1)|=2$.
Thus, there is a vertex $y_1\in Y_0$ adjacent to $a_1, a_3$ and $a_5$. By Lemma \ref{lem1}, $\{a_3a_4,a_4a_5\}\subseteq E(F)$, which means $E_0=\{a_3a_4,a_4a_5\}$, a contradiction as $E_0$ is a matching.
Hence, $E_0\cap E(f_2)\neq \emptyset$ or $E_0\cap E(f_5)\neq\emptyset$.  Without loss of generality, let $E_0\cap E(f_2)\neq \emptyset$ and $e_1\in E(f_2)$.  If $|E(f_2)\cap E(F-G_1)|=1$, i.e., $e_1=a_1a_3$, then $a_4=a_5$ and $a_3a_4\in E(F)$ by Lemma \ref{lem1},
 which means $E_0=\{a_1a_3, a_3a_4\}$, a contradiction as $E_0$ is a matching. Hence, $|E(f_2)\cap E(F-G_1)|=3$.

  \textbf{Case 1.} Vertex $a_1$ is incident with $e_1$. Then, let $f_2\cap (F-G_1)=a_1a_6y_1a_3$, i.e., $e_1=a_1a_6$. Since $E_0$ is a matching,  $a_5$ and $a_6$ have a common neighbor, say $y_2$.

  \begin{figure}[h]
    \centering
      \subfigure[]{
   \label{figf_i}
   \includegraphics[height=32mm]{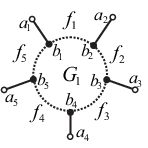}}
    \subfigure[the patch 
    $N_m$]{
   \label{figpH_m}
   \includegraphics[height=32mm]{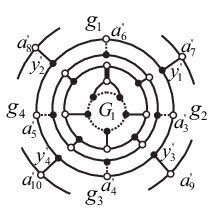}}
    \subfigure[$H_{22}\subseteq F-G_1$]{
   \label{figB22}
   \includegraphics[height=32mm]{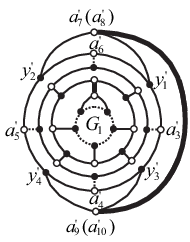}}
    \subfigure[$H_{22}\subseteq F-G_1$]{
    \label{figB23}
    \includegraphics[height=32mm]{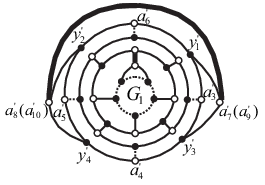}}
    \subfigure[$H_{23}\subseteq F-G_1$]{
    \label{figB24}
    \includegraphics[height=32mm]{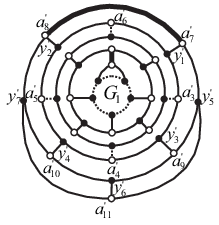}}
    \subfigure[$H_{23}\subseteq F-G_1$]{
    \label{figB25}
    \includegraphics[height=32mm]{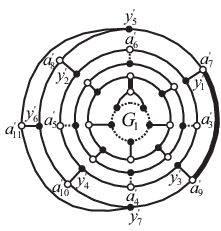}}
    \subfigure[$H_{23}\subseteq F-G_1$]{
    \label{figB26}
    \includegraphics[height=32mm]{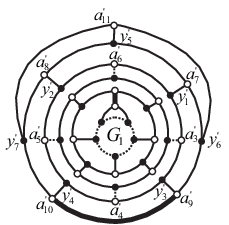}}
    \caption{Illustration for  Subcase 1.1.}
\end{figure}

  \textbf{Subcase 1.1.} Edge $e_2\notin E(f_3)\cup E(f_4)$. Then, $f_i\cap (F-G_1)$ is a path 
  $P_3$ for $i\in\{3,4\}$,
say $a_iy_ia_{i+1}$. Thus, we get a new patch of $F$, say $N_0$. Further, the outer face of $N_0$ is of size 8 with four black 2-degree vertices and four white 3-degree vertices alternating on its facial cycle.
Denote $L_1$ the layer around the patch $N_0$.
By Observation \ref{claim2}(i), we know that $L_1$ consists of exactly four hexagon if each face of $L_1$ is neither  a quadrilateral nor contains $e_2$. Further, by Lemma \ref{lem1}, we can get a new patch $N_1=N_0\cup L_1$ whose outer face is also of size 8 with four black 2-degree vertices and four white 3-degree vertices alternating on its facial cycle if each face of $L_1$ is neither  a quadrilateral nor contains $e_2$. Denote $L_2$ the layer around $N_1$. Similarly, we can also get a new patch $N_2=N_0\cup L_1\cup L_2$ which  has the boundary same with $N_0$ if each face of $L_2$ is neither  a quadrilateral nor contains $e_2$. Thus, we can do this operation repeatedly until the $m$th step such that there is a face of the layer $L_{m+1}$ around the patch $N_m$ being a quadrilateral or containing  $e_2$,  where $N_m=N_0\cup L_1\cup\cdots\cup L_m$.
 Without loss of generality, let $a_6'y_1'a_3'y_3'a_4'y_4'a_5'y_2'a_6'$ be the facial cycle of the outer face of $N_m$. Denote $a_{i+6}'$ the third neighbor of $y_i'$ and $g_i$($1\leqslant i\leqslant4$) the faces of the layer $L_{m+1}$ in the clockwise direction, where $g_1$ is the face along the path $y_2'a_6'y_1'$, see Fig. \ref{figpH_m}.
If $g_1$ is a quadrilateral, then $a_7'=a_8'$. By Lemma \ref{lem1},  $a_9'=a_{10}'$ and $e_2=a_7'a_9'$; see Fig. \ref{figB22}. 
Thus, we have $H_{22}\subseteq F-G_1$.
If $e_2\in E(g_1)$, then $e_2=a_7'a_8'$ by Observation \ref{claim2}(ii). If $g_2$(or $g_4$) is a quadrilateral, then $a_9'=a_7'$ and $a_{10}'=a_8'$; see Fig. \ref{figB23}. 
Thus, we also have $H_{22}\subseteq F-G_1$.
If both $g_2$ and $g_4$ are not quadrilateral, then $a_9'\neq a_{10}'$ by Lemma \ref{lem1} and each of $g_i$($2\leqslant i\leqslant4$) is a hexagon by Observation \ref{claim2}(i). Let $y_{i+3}'$ be the sixth vertices of $g_i$, $2\leqslant i\leqslant4$. Then, 
$H_{23}\subseteq F-G_1$ by Lemma \ref{lem1}; see Fig. \ref{figB24}. If $g_1$ is a hexagon, then $a_7',a_8',a_9'$ and $a_{10}'$ are pairwise different by Lemma \ref{lem1}. Thus, there is vertex $y_5'$ being incident with $a_7'$ and $a_8'$.
If $e_2\in E(g_2)$ or $e_2\in E(g_4)$, without loss of generality, assume $e_2\in E(g_2)$, then $e_2=a_7'a_9'$. 
Further, $a_i'$ and $a_{10}'$ have a common neighbour, say $y_{i-2}'$, $i=8,9$.  By Lemma \ref{lem1}, $y_6'\neq y_7'$. 
Thus, 
$H_{23}\subseteq F-G_1$ by Lemma \ref{lem1}; see Fig. \ref{figB25}. If $e_2\notin E(g_2)\cup E(g_4)$, then $e_2\in E(g_3)$, i.e., $e_2=a_9'a_{10}'$. Thus, 
$a_i'$ and $a_{i+2}'$ have a common neighbour, say $y_{i-1}'$, $i=7,8$. Further, $y_6'\neq y_7'$.
Then, 
$H_{23}\subseteq F-G_1$ by Lemma \ref{lem1}; see Fig. \ref{figB26}.

  \textbf{Subcase 1.2.} Edge $e_2\in E(f_3)\cup E(f_4)$.
  Without loss of generality, let $e_2\in E(f_3)$. Then $f_4\cap (F-G_1)$ is a path 
  $P_3$, say $a_4y_3a_5$. If $e_2=a_3a_4$, then
  $F-G_1=H_{23}$ by Lemma \ref{lem1}; see Fig. \ref{figB11}. If $e_2\neq a_3a_4$, then $|E(f_3)\cap E(F-G_1)|=3$.

\begin{figure}[h]
    \centering
    \subfigure[$F-G_1=H_{23}$]{
   \label{figB11}
   \includegraphics[height=28mm]{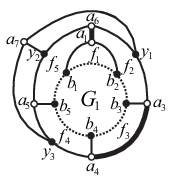}}
     \subfigure[the patch $N_m^1$]{
   \label{figHm}
   \includegraphics[height=28mm]{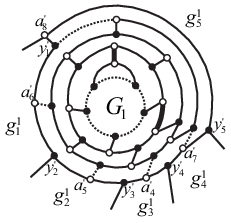}}
         \subfigure[$F-G_1=H_{30}^m$]{
   \label{figB12}
   \includegraphics[height=26mm]{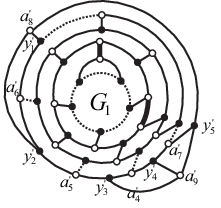}}
   \subfigure[$F-G_1=H_{31}^m$]{
   \label{figB13}
   \includegraphics[height=26mm]{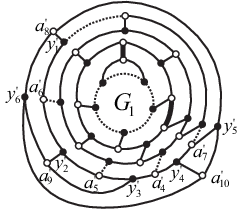}}
   \subfigure[$F-G_1=H_{32}^m$]{
   \label{figB14}
   \includegraphics[height=28mm]{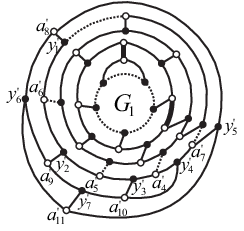}}
     \subfigure[$F-G_1=H_{33}^m$]{
   \label{figB15}
   \includegraphics[height=28mm]{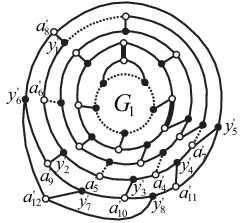}}
   \subfigure[$F-G_1=H_{34}^m$]{
   \label{figB16}
   \includegraphics[height=28mm]{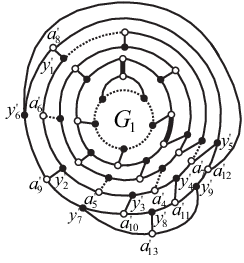}}
    \caption{Illustration for Subcase 1.2.}
\end{figure}

 If $a_3$ is incident with $e_2$, then let $f_3\cap (F-G_1)=a_3a_7y_4a_4$. Let $a_8$ and $y_5$ be the third neighbors of $y_1$ and $a_7$ respectively. Then, $a_8\neq a_7$ and $y_5\notin \{y_2,y_3,y_4\}$ by Lemmas \ref{lem1} and \ref{lem2}. Thus, $a_8y_5\in E(F)$ by Observation \ref{claim2}(ii).  Hence, we get a new patch $N_0^1$. Further,  the outer face of $N_0^1$ is of size 10 and the degrees of vertices on the boundary of $N_0^1$ form a degree-array [2,2,3,2,3,2,3,2,3,3] with one white 2-degree vertex and four black 2-degree vertices which is read from vertex $a_8$ in the clockwise direction.
  Note that each face of $F$ but not $N_0^1$ is either a quadrilateral or a hexagon by Observation \ref{claim2}(i). Denote $L_1$ the layer around $N_0^1$. Then, by Lemma \ref{lem1}, we can get a new patch $N_1^1=N_0^1\cup L_1$ which has the boundary same with $N_0^1$ if each face of $L_1$ is a hexagon. Let $L_2$ be the layer around the patch $N_1^1$. Similarly, we can also get a new patch $N_2^1=N_0^1\cup L_1\cup L_2$ which has the boundary same with $N_0^1$ if each face of $L_2$ is a hexagon.
  Thus, we can do the this operation repeatedly until the $m$th step such that the layer $L_{m+1}$ around $N_m^1$ contains a quadrilateral, where $N_m^1=N_0^1\cup L_1\cup \cdots\cup L_m$. Let $a_8'y_5'a_7'y_4'a_4'y_3'a_5'y_2'a_6'y_1'a_8'$ be the facial cycle of the outer face of $N_m^1$ and $g_i^1$($1\leqslant i\leqslant5$) the faces of $L_{m+1}$ in the counterclockwise direction, where $g_1^1$ is the face along the path $a_8'y_1'a_6'y_2'$; see Fig. \ref{figHm}.  If $g_1^1$ is a quadrilateral, then $a_8'y_2'\in E(F)$. Thus,  $F-G_1=H_{30}^m$ by Lemma \ref{lem1}; see Fig. \ref{figB12}. If $g_1^1$ is a hexagon, then there is a path $a_8'y_6'a_9'y_2'$ by the fact that $F-H_0^1$ is bipartite. If $g_2^1$ is a quadrilateral, then $a_9'y_3'\in E(F)$.  Thus, $F-G_1=H_{31}^m$ by Lemma \ref{lem1}; see Fig. \ref{figB13}. If $g_2^1$ is a hexagon, then there is a path $a_9'y_7'a_{10}'y_3'$.
  Thus, by a reasoning completely analogous to the analysis applied to $g_1^1$, we analysis the cases that the faces $g_3^1, g_4^1$ and $g_5^1$ are quadrilaterals in turn and get $F-G_1$ is isomorphic to $H_{32}^m, H_{33}^m$ and $H_{34}^m$ respectively; see Figs. \ref{figB14}-\ref{figB16}.

 If $a_4$ is incident with $e_2$, then we can assume $f_3 \cap(F-G_1) =a_3y_4a_7a_4$. Let $a_8$ and $y_5$ be the third neighbors of $y_3$ and $a_7$ respectively.  Thus, $a_8\neq a_7$ and $y_5\notin \{y_1,y_2,y_4\}$  by Lemmas \ref{lem1} and \ref{lem2}.
 Thus $a_8y_5\in E(F)$ by Observation \ref{claim2}(ii). Hence, we get a new patch $N_0^2$. Further, the outer face of $N_0^2$ is of size 10 and the degrees of vertices on the boundary of $N_0^2$ has a degree-array [2,2,3,2,3,2,3,2,3,3] with one white 2-degree vertices and four black 2-degree vertices which is read from $a_8$ in the counterclockwise direction.  Clearly, the boundary of $N_0^2$ is isomorphic to the boundary of $N_0^1$.  By the similar analysis of $F-N_0^1$, we can get $F-G_1\in \{H_{35}^m, \ldots, H_{39}^m\}$.

\textbf{Case 2.} Vertex $a_3$ is incident with $e_1$. Then, let $f_2\cap (F-G_1)$ is $a_1y_1a_6a_3$, i.e., $e_1=a_6a_3$.
 If $e_2\notin \cup_{i=3}^5E(f_i)$, then  $y_1a_5\in E(F)$ and $f_i\cap (F-G_1)$ is a path 
 $P_3$ for  $i\in\{3,4\}$, say $a_iy_{i-1}a_{i+1}$. Thus, by Lemma \ref{lem1},  $a_6, y_2$ and $y_3$ have a common neighbor $a_7$ as $F_0'$ is bipartite.
 That is, $e_2=a_6a_7$, a contradiction as $E_0$ is a matching. Hence, $e_2\in \cup_{i=3}^5E(f_i)$.

 \begin{figure}[h]
    \centering
   \subfigure[
   $H_{22}\subseteq F-G_1$]{
   \label{figB27}
   \includegraphics[height=29mm]{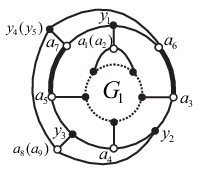}}
    \subfigure[$H_{24}\subseteq F-G_1$]{
    \label{figB28}
    \includegraphics[height=30mm]{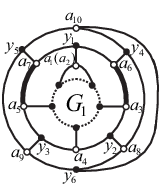}}
  \subfigure[
  $H_{23}\subseteq F-G_1$]{
   \label{figB29}
   \includegraphics[height=29mm]{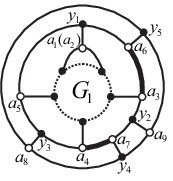}}
    \subfigure[
    $H_{22}\subseteq F-G_1$]{
    \label{figB30}
    \includegraphics[height=29mm]{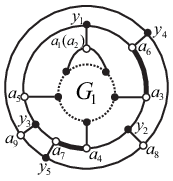}}
    \subfigure[
     $H_{22}\subseteq F-G_1$]{
   \label{figB31}
   \includegraphics[height=29mm]{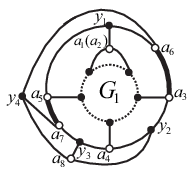}}
    \caption{Illustration for Case 2.}
\end{figure}

 If $e_2\in E(f_5)$, then 
  $a_5$ and $y_1$ have a common neighbor $a_7$ and
  $f_i\cap (F-G_1)$ is a path 
  $P_3$ for $i=3,4$, say $a_iy_{i-1}a_{i+1}$. 
  Let $y_4, y_5, a_8$ and $a_9$ be the third neighbors of $a_6,a_7,y_2$ and $y_3$ respectively. By Lemmas 
 \ref{lem1} and \ref{lem2}, $y_i\notin\{y_2,y_3\}$ and $a_{i+4}\notin\{a_6,a_7\}$, $i=4,5$. If $y_4=y_5$, then $a_8=a_9$ and $a_8y_4\in E(F)$ by Lemma \ref{lem1}. That is, 
  $H_{22}\subseteq F-G_1$; see Fig. \ref{figB27}.
If $y_4\neq y_5$, then $a_8\neq a_9$ by Lemma \ref{lem1}. Hence, $\{y_4a_8, y_5a_9\}\subseteq E(F)$ by Observation \ref{claim2}(ii). Moreover, $y_4$ and $y_5$ have a common neighbor $a_{10}$.
Thus, $H_{24}\subseteq F-G_1$ by Lemma \ref{lem1}; see Fig. \ref{figB28}. If $e_2\notin E(f_5)$, then $y_1a_5\in E(F)$.

If $e_2\in E(f_3)$, then $e_2\neq a_3a_4$ and we can assume $f_3\cap (F-G_1)=a_3y_2a_7a_4$ as $E_0$ is a matching. Thus, $f_4\cap (F-G_1)$ is a path 
$P_3$, say $a_4y_3a_5$. Let $a_8$ and $y_4$ be the third neighbors of $y_3$ and $a_7$ respectively. Then, $a_8\notin \{a_6,a_7\}$ and $y_4\neq y_3$ by Lemma \ref{lem1}. Thus, $a_8y_4\in E(F)$ by Observation \ref{claim2}(ii) while $a_8$ and $a_6$ have a common neighbor $y_5$ by Observation \ref{claim2}(i). Thus,  
by Lemma \ref{lem1}, $H_{23}\subseteq F-G_1$; see Fig. \ref{figB29}. If $e_2\notin E(f_3)$, then let $f_3\cap (F-G_1)=a_3y_2a_4$. 

If $e_2\in E(f_4)$, then $e_2\neq a_4a_5$ by Lemmas \ref{lem1}. If $a_4$ is incident with $e_2$, then let $f_4\cap (F-G_1)=a_4a_7y_3a_5$.  Let $y_4,y_5$ and $a_8$ be the third neighbors of $a_6,a_7$ and $y_2$ respectively. By Lemmas \ref{lem1} and \ref{lem2}, $y_2,y_3,y_4$ and $y_5$ are pairwise different.     Thus, $\{y_4a_8,a_8y_5\}\subseteq E(F)$ by Observation \ref{claim2}(ii). Hence, 
 $H_{22}\subseteq F-G_1$ by Lemma \ref{lem1}; see Fig. \ref{figB30}. If $a_5$ is incident with $e_2$, then let  $f_4\cap (F-G_1)=a_4y_3a_7a_5$. Denote $y_4$ and $a_8$ the third neighbors of $a_6$ and $y_2$ respectively. By Lemma \ref{lem1}, $y_4\neq y_3$ and $a_8\neq a_7$. Thus, $\{y_4a_7, y_4a_8\}\subseteq E(F)$ by Observation \ref{claim2}(ii). Then, $y_3a_8\in E(F)$ by Lemma \ref{lem1}, i.e., 
$H_{22}\subseteq F-G_1$; see Fig. \ref{figB31}.

\section{The Proof of Fact 2 }

Let $|E(f_1)\cap E(F-G_1)|=1$. Then $a_1a_2\in E_0$.
Without loss of generality, let $e_1=a_1a_2$. Further, $a_1, \ldots, a_4$ and $a_5$ are pairwise different as $|E(f_1)\cap E(F-G_1)|\leqslant|E(f_i)\cap E(F-G_1)|$ for $2\leqslant i\leqslant5$.

 \begin{figure}[h]
    \centering
   \subfigure[$H_{22}\subseteq F-G_1$]{
   \label{figB32}
   \includegraphics[height=27mm]{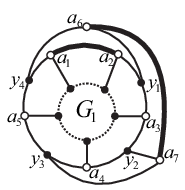}}
    \subfigure[$F-G_1=H_{25}$]{
   \label{figB32a}
   \includegraphics[height=27mm]{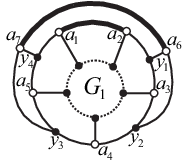}}
    \subfigure[$H_{23}\subseteq F-G_1$]{
   \label{figB32b}
   \includegraphics[height=27mm]{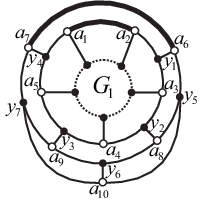}}
    \subfigure[$H_{25}\subseteq F-G_1$]{
    \label{figB33}
    \includegraphics[height=27mm]{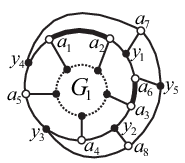}}
    \subfigure[$F-G_1=H_{26}$]{
   \label{figB34}
   \includegraphics[height=27mm]{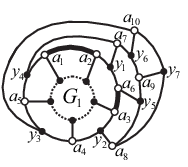}}
     \subfigure[$F-G_1=H_{22}$]{
   \label{figB35(a)}
   \includegraphics[height=27mm]{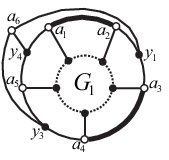}}
    \subfigure[$F-G_1=H_{27}$]{
   \label{figB35}
   \includegraphics[height=27mm]{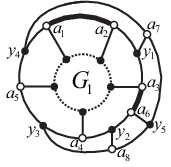}}
    \subfigure[
    $H_{23}\subseteq F-G_1$]{
    \label{figB36}
    \includegraphics[height=27mm]{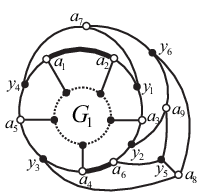}}
    \caption{Illustration for Fact 2. 
    }
\end{figure}

For any $2\leqslant i\leqslant5$, if $e_2\notin E(f_i)$, then $f_i\cap (F-G_1)$ is a path 
$P_3$, say $a_iy_{i-1}a_{i+1}$, where $a_6=a_1$. By the 3-regularity of $F$, $y_1,y_2,y_3$ and $y_4$ are pairwise different. If $e_1$ and $e_2$ do not lie on the same face of $F$, then 
$y_1$ and $y_4$ have a common neighbor by Observation \ref{claim2}(ii), say $a_6$. Then, by Lemma \ref{lem1}, there is a vertex $a_7$ adjacent to $a_6, y_2$ and $y_3$ as $F_0'$ is bipartite. Thus, $e_2=a_6a_7$ and 
$H_{22}\subseteq F-G_1$; see Fig. \ref{figB32}. If $e_1$ and $e_2$  lie on the same face of $F$, then the face containing $e_1$ and $e_2$ is a hexagon  by Observation \ref{claim2}(iii). Let $a_6$ and $a_7$ be the third neighbors of $y_1$ and $y_4$  respectively. Then $e_2=a_6a_7$.  By Lemma \ref{lem1}, $\{a_6y_3,y_2a_7\}\nsubseteq E(F)$. If $a_6y_2\in E(F)$, then $a_7y_3\in E(F)$. Thus, $F-G_1=H_{25}$; see Fig. \ref{figB32a}.
If $a_6y_2\notin E(F)$, then there is a path $a_6y_5a_8y_2$. By Lemma \ref{lem1} and \ref{lem2}, $\{a_8y_3, y_3a_7\}\nsubseteq E(F)$. Thus, there are two paths $a_8y_6a_9y_3$ and $a_9y_7a_7$. By
Lemma \ref{lem1}, $y_5,y_6$ and $y_7$ are pairwise different. Hence, 
$H_{23}\subseteq F-G_1$ by Lemma \ref{lem1}; see Fig. \ref{figB32b}.

If $e_2\in E(f_2)$ or $e_2\in E(f_5)$, without loss of generality, let $e_2\in E(f_2)$, then let $f_2\cap (F-G_1)=a_2y_1a_6a_3$ as $E_0$ is a matching. 
Further,  $f_i\cap (F-G_1)$ is a path 
$P_3$, say $a_iy_{i-1}a_{i+1}$, where $3\leqslant i\leqslant5$ and $a_6=a_1$.  By Lemmas \ref{lem1} and \ref{lem2}, $y_1,y_2,y_3$ and $y_4$ are pairwise different.  Since $F$ is 3-regular, any two faces have at most one common edge. Thus, $e_1$ and $e_2$ do not lie on the same face of $F$. By Observation \ref{claim2}(ii), $y_1$ and $y_4$ have a common neighbor, say $a_7$. Clearly, $a_7\neq a_6$.  Let $y_5$ and $a_8$ be the third neighbors of $a_6$ and $y_2$. Then, $y_5\notin\{y_2,y_3\}$ and $a_8\neq a_7$. 
Thus, $y_5a_8\in E(F)$ by Observation \ref{claim2}(ii).  If $y_5a_7\in E(F)$, then $a_8y_3\in E(F)$ and 
$H_{25}\subseteq F-G_1$; see Fig. \ref{figB33}.  If $y_5a_7\notin E(F)$, then there is a path $y_5a_9y_6a_7$ by Observation \ref{claim2}(i). Further, $y_3$ and $y_6$ have a common neighbor $a_{10}$. Thus, $F-G_1=H_{26}$ by Lemma \ref{lem1}; see Fig. \ref{figB34}.

If $e_2\in E(f_3)$ or $e_2\in E(f_4)$, without loss of generality, let $e_2\in E(f_3)$, then $f_i\cap (F-G_1)$ is a path 
$P_3$, say $a_iy_{i-1}a_{i+1}$, where  for $i\in\{2,4,5\}$ and $a_6=a_1$. By the 3-regularity of $F$, $y_1, y_3$ and $y_4$ are pairwise different. If $e_2=a_3a_4$, then $F-G_1=H_{22}$ by Lemma \ref{lem1}; see Fig. \ref{figB35(a)}. If $e_2\neq a_3a_4$, then $|E(f_3)\cap E(F-G_1)|=3$.
 If $a_3$ is incident with $e_2$, then let $f_3\cap (F-G_1)=a_3a_6y_2a_4$. Clearly, $\{y_1a_6, y_4a_6\}\nsubseteq E(F)$. Thus,
 $y_1$ and $y_4$ have a common neighbor, say $a_7$.
Further, $a_6$ and $a_7$ have a common neighbor $y_5$.
Thus,  $F-G_1=H_{27}$ by Lemma \ref{lem1}; see Fig. \ref{figB35}.
 If $a_4$ is incident with $e_2$, then let $f_3\cap (F-G_1)=a_3y_2a_6a_4$. Obviously, $\{y_1a_6,y_4a_6\}\nsubseteq E(F)$. Thus, $y_1$ and $y_4$ have a common neighbour, say $a_7$. Let $a_8$ be the third neighbor of $y_3$. Then, $a_8\notin\{a_6,a_7\}$.
 Thus, $a_6$ and $a_8$  have a common neighbor $y_5$ by Observation \ref{claim2}(ii) while $a_7$ and $a_8$ have a common neighbor $y_6$  by Observation \ref{claim2}(i). Hence, 
 $H_{23}\subseteq F-G_1$ by Lemma \ref{lem1}; see Fig. \ref{figB36}.

\section{The Proof of Fact 3 }

Let $|E(f_1)\cap E(F-G_1)|=2$. Then $E_0\cap E(f_1)=\emptyset$.
Without loss of generality, let $f_1\cap (F-G_1)=a_1y_1a_2$.
Since $|E(f_1)\cap E(F-G_1)|\leqslant|E(f_i)\cap E(F-G_1)|$, $|E(f_i)\cap E(F-G_1)|\geqslant2$, where $2\leqslant i\leqslant5$. That is, $e_2\neq a_ia_{i+1}$, $2\leqslant i\leqslant5$ and $a_6=a_1$.

\textbf{Case 1.} $E_0\cap E(f_i)=\emptyset$ for any $2\leqslant i\leqslant5$.

Then, for each $2\leqslant i\leqslant5$, $f_i\cap (F-G_1)$ is a path 
$P_3$,  say $a_iy_ia_{i+1}$, where $a_6=a_1$.
Thus, we get a new patch $N_0^3$ and  the outer face of $N_0^3$ is of size 10 with five black 2-degree vertices and five white 3-degree vertices alternating on its facial cycle.
Denote $\bar{N}_0^3=F-N_0^3$. By Lemmas \ref{lem1} and \ref{lem2}, $\bar{N}_0^3$ has no isolated vertex. We claim $\bar{N}_0^3$ has no pendent vertex.  Otherwise, assume $a_6$ is a pendent vertex of $\bar{N}_0^3$. Then, by Lemma \ref{lem1}, $a_6$ must be adjacent to two successive 2-degree vertices of $N_0^3$, say $y_1$ and $y_2$. Let $a_7$ be the third neighbor of $y_3$. Clearly, $a_7\neq a_6$.
 If $a_6a_7\in E(F)$,  then, by Lemma \ref{lem1}, there is a vertex $a_8$ adjacent to $y_5, y_4$ and $a_7$, i.e., $\{a_6a_7,a_7a_8\}=E_0$,
 a contradiction as $E_0$ is a matching. If $a_6a_7\notin E(F)$, then  $a_6$ and $a_7$ have a common neighbor $y_6$ by Observation \ref{claim2}(i). Further, $y_5$ and $y_6$ have a common neighbor $a_8$.  Thus, by Lemma \ref{lem1}, $a_7, a_8$ and $y_4$ have a common neighbor $a_9$, i.e., $E_0=\{a_7a_9,a_8a_9\}$, a contradiction as $E_0$ is a matching. Our claim holds.
 Thus, the layer $L_1$ around $N_0^3$ does not contain quadrilaterals.
 By Observation \ref{claim2}(i), each face $f$ of $L_1$ is a hexagon if $E_0\cap E(f)=\emptyset$.
Further, by Lemma \ref{lem1}, we can get a new patch $N_1^3=N_0^3\cup L_1$ which has the boundary same with $N_0^3$ if $E_0\cap E(L_1)=\emptyset$. Denote $L_2$ the layer around the patch $N_1^3$. Then, similar as the above analysis, we can also get a new patch $N_2^3=N_0^3\cup L_1\cup L_2$ has the boundary same with $N_0^3$ if $E_0\cap E(L_2)=\emptyset$.
Thus, we can repeatedly do this operation until the $m$th step such that 
$E_0\cap E(L_{m+1})\neq \emptyset$.
Let $C=a_1'y_1'a_2'y_2'a_3'y_3'a_4'y_4'a_5'y_5'a_1'$ be the facial cycle of the outer face of $N_m^3$, where $N_m^3=N_0^3\cup L_1 \cup\cdots\cup L_m$; see Fig. \ref{Hm3}.

 \begin{figure}[h]
    \centering
     \subfigure[the patch $N_m^3$]{
    \label{Hm3}
    \includegraphics[height=27mm]{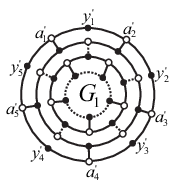}}
   \subfigure[
  $H_{22}\subseteq F-G_1$]{
   \label{figC1}
   \includegraphics[height=27mm]{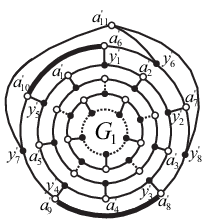}}
    \subfigure[
  $H_{22}\subseteq F-G_1$]{
    \label{figC2}
    \includegraphics[height=27mm]{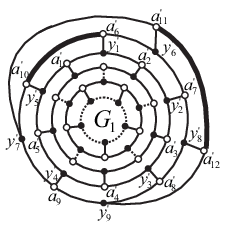}}
     \subfigure[
      $H_{25}\subseteq F-G_1$]{
   \label{figC5}
   \includegraphics[height=27mm]{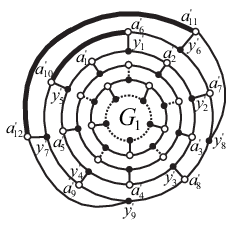}}
    \subfigure[
    $H_{22}\subseteq F-G_1$]{
    \label{figC4}
    \includegraphics[height=27mm]{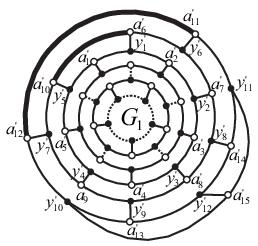}}
    \caption{Illustration for  Case 1.}
\end{figure}
Similar as the analysis of $\bar{N}_0^3$,  $\bar{N}_m^3=F-N_m^3$  has neither isolated vertices nor pendent vertices.
Let $a_{i+5}'$ be the third neighbor of $y_i'$ respectively, $1\leqslant i\leqslant5$. Then, $a_6',\ldots, a_9'$ and $a_{10}'$ are pairwise different. Since $E_0\cap E(L_{m+1})\neq \emptyset$, without loss of generality, let $e_1=a_{10}'a_6'$.  Since $E_0$ is a matching, $a_6'$ and $a_7'$ have a common neighbor $y_6'$ while $a_9'$ and $a_{10}'$ have a common neighbor $y_7'$.
If $e_2=a_7'a_8'$ or $e_2=a_8'a_9'$, without loss of generality, let $e_2=a_8'a_9'$,  then there is a vertex $y_8'$ being adjacent to $a_7'$ and $a_8'$. Thus,  
$H_{22}\subseteq F-G_1$ by Lemma \ref{lem1}; see Fig. \ref{figC1}.
If $e_2\neq a_i'a_{i+1}'$, then there is a vertex, say $y_{i+1}'$, adjacent to $a_i'$ and $a_{i+1}'$, $i=7,8$.  If $e_2$ and $e_1$ do not lie on the same face of $F$, then $y_6'$ and $y_7'$ have a common neighbor $a_{11}'$ by Observation \ref{claim2}(ii). Thus, by Lemma \ref{lem1}, there is a vertex $a_{12}'$ adjacent to $a_{11}', y_8'$ and $y_9'$; 
see Fig. \ref{figC2}. Thus, $e_2=a_{11}'a_{12}'$ and 
$H_{22}\subseteq F-G_1$.
If $e_2$ and $e_1$ lie on the same face of $F$, then, by Observation \ref{claim2}(iii), the third neighbors of $y_6'$ and $y_7'$, say $a_{11}'$ and $a_{12}'$ respectively, are adjacent, i.e., $e_2=a_{11}'a_{12}'$. If $a_{12}'y_9'\in E(F)$, then $a_{11}'y_8'\in E(F)$, i.e., 
$H_{25}\subseteq F-G_1$; see Fig. \ref{figC5}.  If $a_{12}'y_9'\notin E(F)$, then there is a path $a_{12}'y_{10}'a_{13}'y_9'$ by Observation \ref{claim2}(i).
Let $a_{14}'$ be the third neighbor of $y_8'$.  
Then, $a_{14}'\notin\{a_{11}', a_{13}'\}$. Thus, $a_{11}'$ and $a_{14}'$ have a common neighbor $y_{11}'$ while $a_{13}'$ and $a_{14}'$ have a common neighbor $y_{12}'$. 
Hence,  $H_{22}\subseteq F-G_1$ by Lemma \ref{lem1}; see Fig. \ref{figC4}.

\textbf{Case 2.} $E_0\cap E(f_i)\neq\emptyset$ for some $2\leqslant i\leqslant5$.   By the planarity of $F$,
we can assume  $e_1\in E(f_2)$ and let $f_2\cap (F-G_1)=a_2a_6y_2a_3$. Further, there are at most one subscript $3\leqslant i\leqslant5$ such that $|E(f_i)\cap E(F-G_1)|=3$ as $|E_0|=2$.

\textbf{Subcase 2.1.} Edge $e_2\in E(f_3)$. Then, for $i\in \{4, 5\}$, we can assume $f_i\cap (F-G_1)=a_iy_ia_{i+1}$, where $a_6=a_1$.

\textbf{Subcase 2.1.1.} If vertex $a_3$ is incident with $e_2$, then let $f_3\cap (F-G_1)=a_3a_7y_3a_4$.

Let $a_8,a_9,y_6$ and $y_7$ be the third neighbors of $y_1,y_2, a_6$ and $a_7$ respectively. By Lemmas \ref{lem1} and \ref{lem2},
any two vertices of $\{a_6,a_7,a_8,a_9,y_1,\ldots,y_7\}$ are different.
 Thus, $\{a_8y_6,a_9y_7\}\subseteq E(F)$  by Observation \ref{claim2}(ii) and  we get a new patch $N_0^4$.
  Further, the outer face of $N_0^4$ is of size 14 and the degrees of the vertices on its boundary form a degree-array [2,2,3,3,2,2,3,2,3,2,3,2,3,3] with two white 2-degree vertices and five black 2-degree vertices which is read form $a_8$ in the clockwise direction.
 By Observation \ref{claim2}(ii), every face of $F$ but not $N_0^4$ is either a quadrilateral or a hexagon. Denote $L_1$ the layer around  $N_0^4$. By Lemma \ref{lem1},
 we can get a new patch $N_1^4=N_0^4\cup L_1$ which has the boundary same with $N_0^4$ if each face of $L_1$ is a hexagon.
Denote 
 $L_2$ the layer around $N_1^4$. Similarly, 
  we can also get a new patch $N_2^4=N_0^1\cup L_1\cup L_2$ which has the boundary same with $N_0^4$ if each face of $L_2$ is a hexagon. 
Thus, we can do this operation repeatedly until the $m$th step such that the layer $L_{m+1}$ around $N_m^4$ contains a quadrilateral, where $N_m^4=N_0^4\cup L_1\cup\cdots\cup L_m$. Let $a_1'y_1'a_8'y_6'a_6'y_2'a_9'y_7'a_7'y_3'a_4'y_4'a_5'y_5'a_1'$ be the facial cycle of the outer face of $N_m^4$ and $g_i^4$($1\leqslant i\leqslant7$) the faces of the layer $L_{m+1}$ in the counterclockwise direction, where $g_1^4$ is the face along the path $a_8'y_1'a_1'y_5'$; see Fig. \ref{H_0}.

If $g_1^4$ is a quadrilateral, then $a_8'y_5'\in E(F)$. Further, $y_4'$ and $y_6'$ have a common neighbor $a_{10}'$ while $a_9'$ and $a_{10}'$  have a common neighbor $y_8'$. Thus,  $F-G_1=H_{40}^m$ by Lemma \ref{lem1}; see Fig. \ref{figB79}.  If $g_1^4$ is a hexagon, then there is a path $a_8'y_8'a_{10}'y_5'$ and $a_{10}'\neq a_9'$ by Lemma \ref{lem1}.

If $g_2^4$ is a quadrilateral, then $a_{10}'y_4'\in E(F)$. Further,  $y_8'$ and $y_3'$ have a common neighbor $a_{11}'$. If $a_{11}'y_7'\in E(F)$, then $y_6'a_9'\in E(F)$. Thus, $F-G_1=H_{41}^m$; see Fig. \ref{figB80}. If $a_{11}'y_7'\notin E(F)$, then there is a path $a_{11}'y_9'a_{12}'y_7'$.
Moreover, $y_9'$ and $y_6'$ have a common neighbor $a_{13}'$. Thus, $F-G_1=H_{42}^m$ by Lemma \ref{lem1}; see Fig. \ref{figB81}. If
$g_2^4$ is a hexagon, then there is a path $a_{10}'y_9'a_{11}'y_4'$ by Lemma \ref{lem1}.

If $g_3^4$  is a quadrilateral, then $a_{11}'y_3'\in E(F)$. Further, $y_9'$ and $y_7'$  have a common neighbor $a_{12}'$.
If $a_{12}'y_8'\in E(F)$, then $y_6'a_9'\in E(F)$.  Thus, $F-G_1=H_{43}^m$; see Fig. \ref{figB82}. 
 \begin{figure}[h]
    \centering
   \subfigure[the patch $N_m^4$]{
   \label{H_0}
   \includegraphics[height=32mm]{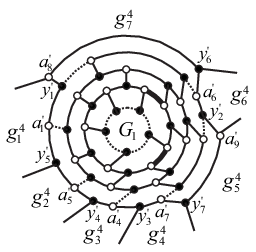}}
    \subfigure[$F-G_1=H_{40}^m$]{
    \label{figB79}
    \includegraphics[height=32mm]{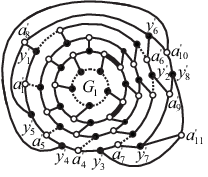}}
     \subfigure[$F-G_1=H_{41}^m$]{
   \label{figB80}
   \includegraphics[height=32mm]{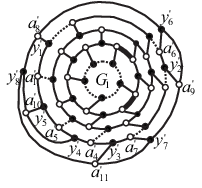}}
    \subfigure[$F-G_1=H_{42}^m$]{
    \label{figB81}
    \includegraphics[height=32mm]{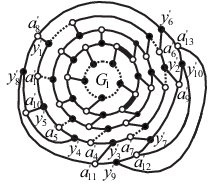}}
     \subfigure[$F-G_1=H_{43}^m$]{
   \label{figB82}
   \includegraphics[height=32mm]{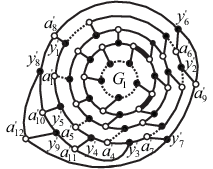}}
    \subfigure[$F-G_1=H_{44}^m$]{
    \label{figB83}
    \includegraphics[height=32mm]{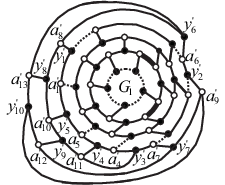}}
     \subfigure[$F-G_1=H_{45}^m$]{
   \label{figB84}
   \includegraphics[height=32mm]{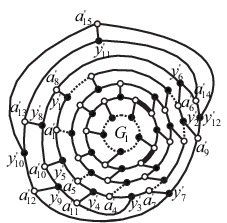}}
    \subfigure[$F-G_1=H_{46}^m$]{
    \label{figB85}
    \includegraphics[height=32mm]{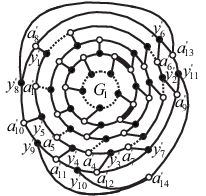}}
     \subfigure[$F-G_1=H_{47}^m$]{
   \label{figB86}
   \includegraphics[height=32mm]{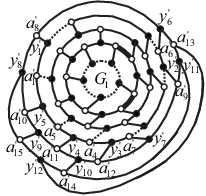}}
    \subfigure[$F-G_1=H_{48}^m$]{
    \label{figB87}
    \includegraphics[height=32mm]{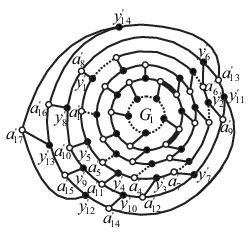}}
     \subfigure[$F-G_1=H_{49}^m$]{
   \label{figB88}
   \includegraphics[height=32mm]{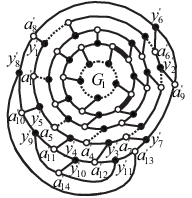}}
    \subfigure[$F-G_1=H_{50}^m$]{
    \label{figB89}
    \includegraphics[height=32mm]{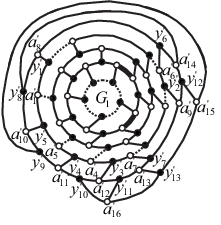}}
     \subfigure[$F-G_1=H_{51}^m$]{
   \label{figB90}
   \includegraphics[height=32mm]{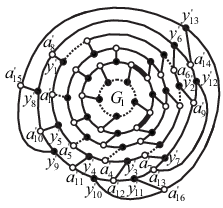}}
    \subfigure[$F-G_1=H_{52}^m$]{
    \label{figB91}
    \includegraphics[height=32mm]{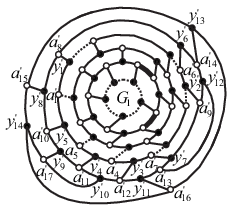}}
     \subfigure[$F-G_1=H_{53}^m$]{
   \label{figB92}
   \includegraphics[height=32mm]{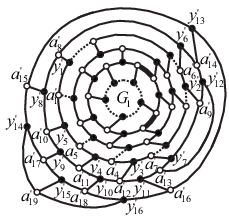}}
   \caption{Illustration for Subcase 2.1.1.}
   \end{figure}
 If $a_{12}'y_8'\notin E(F)$, then there is a path $a_{12}'y_{10}'a_{13}'y_8'$. If $a_{13}'y_6'\in E(F)$, then $y_{10}'a_9'\in E(F)$. Hence, $F-G_1=H_{44}^m$; see Fig. \ref{figB83}. If $a_{13}'y_6'\notin E(F)$, then there is a path $a_{13}'y_{11}'a_{14}'y_6'$. Thus, $a_{14}'$ and $a_9'$ have a common neighbor $y_{12}'$. Hence, $F-G_1=H_{45}^m$ by Lemma \ref{lem1}; see Fig. \ref{figB84}. If $g_3^4$ is a quadrilateral, then there is a path $a_{11}'y_{10}'a_{12}'y_3'$.

If $g_4^4$ is a quadrilateral, then $a_{12}'y_7'\in E(F)$. If $y_{10}'a_9'\in E(F)$, then the face along the path $y_9'a_{11}'y_{10}'a_9'y_2'a_6'y_6'$ is of size at least 8, a contradiction.  Hence, $y_{10}'a_9'\notin E(F)$ and there is a path $y_{10}'a_{14}'y_{11}'a_9'$. Further, $y_6'$ and $y_{11}'$ have a common neighbor $a_{13}'$. If $y_9'a_{14}'\in E(F)$, then $y_8'a_{13}'\in E(F)$. Thus, $F-G_1=H_{46}^m$; see Fig. \ref{figB85}. If $y_9'a_{14}'\notin E(F)$, then there is a path $y_9'a_{15}'y_{12}'a_{14}'$. If $a_{15}'y_8'\in E(F)$, then $y_{12}'a_{13}'\in E(F)$. That is, $F-G_1=H_{47}^m$; see Fig. \ref{figB86}. If $a_{15}'y_8'\notin E(F)$, then there is a path $a_{15}'y_{13}'a_{16}'y_8'$. Further, $a_{16}'$ and $a_{13}'$ have a common neighbor $y_{14}'$. Thus, $F-G_1=H_{48}^m$ by Lemma \ref{lem1}; see Fig. \ref{figB87}. If $g_4^4$ is a hexagon, then there is a path $a_{12}'y_{11}'a_{13}'y_7'$.

If $g_6^4$ is a quadrilateral, then $y_6'a_9'\in E(F)$. Further, $a_{13}'y_8'\in E(F)$. Thus, $F-G_1=H_{49}^m$ by Lemma \ref{lem1}; see Fig. \ref{figB88}. If $g_6^4$ is a hexagon, then there is a path $y_6'a_{14}'y_{12}'a_9'$.

If $g_7^4$ is a quadrilateral, then $y_8'a_{14}'\in E(F)$.  Further, $y_9'$ and $y_{12}'$ have a common neighbor $a_{15}'$ while $a_{15}'$ and $a_{13}'$ have a common neighbor $y_{13}'$. Thus, by Lemma \ref{lem1}, $F-G_1=H_{50}^m$; see Fig. \ref{figB89}. If $g_7^4$ is a hexagon, then there is a path $a_{14}'y_{13}'a_{15}'y_8'$.

Since the layer $L_{m+1}$ contains a quadrilateral, $g_5^4$ must be a quadrilateral, i.e., $a_{13}'y_{12}'\in E(F)$. Thus, $y_{11}'$ and $y_{13}'$ have a common neighbor $a_{16}'$. If $y_9'a_{15}'\in E(F)$, then $y_{10}'a_{16}'\in E(F)$. That is,  $F-G_1=H_{51}^m$; see Fig. \ref{figB90}.  If $y_9'a_{15}'\notin E(F)$, then there is a path $y_9'a_{17}'y_{14}'a_{15}'$. If $y_{10}'a_{17}'\in E(F)$, then $y_{14}'a_{16}'\in E(F)$. Hence, $F-G_1=H_{52}^m$; see Fig. \ref{figB91}. If $y_{10}'a_{17}'\notin E(F)$, then there is a path $y_{10}'a_{18}'y_{15}'a_{17}'$. Further, $a_{16}'$ and $a_{18}'$ have a common neighbor $y_{16}'$. Thus, $F-G_1=H_{53}^m$ by Lemma \ref{lem1}; see Fig. \ref{figB92}.

\textbf{Subcase 2.1.2.} If $a_4$ is incident with $e_2$, then let $f_3\cap (F-G_1)=a_3y_3a_7a_4$. Let $a_8,a_9,y_6$ and $y_7$ be the third neighbors of $y_1,y_4, a_6$ and $a_7$. By Lemmas \ref{lem1} and \ref{lem2},
any two vertices of $\{a_6, a_7, a_8, a_9, y_1,y_2, \ldots, y_7\}$ are different.
Thus, $\{a_8y_6, a_9y_7\}\subseteq E(F)$ by Observation \ref{claim2}(ii). Let $a_{10}$ be the third neighbor of $y_5$.
If $a_{10}=a_8$(resp. $a_{10}=a_9$), then $y_6a_9\in E(F)$(resp. $y_7a_8\in E(F)$). Thus, by Lemma \ref{lem1}, $F-G_1=H_{28}$; see Fig. \ref{figB93}.
If $a_{10}\notin\{a_8,a_9\}$, then $a_i$ and $a_{10}$ have a common neighbor, say $y_i$, $i=8,9$. Denote the patch by $N_0^5$. Then, the outer face of $N_0^5$ is of size 12 with six white 3-degree vertices and six black 2-degree vertices alternating on its facial cycle.
By Observation \ref{claim2}(i), every face of $F$ but not $N_0^5$ is either a quadrilateral or a hexagon. Denote $L_1$ the layer around  $N_0^5$.
 By Lemmas \ref{lem1} and \ref{lem2}, we can get a new patch $N_1^5=N_0^5\cup L_1$ which has the boundary same with $N_0^5$ if
each face of the layer $L_1$ is a hexagon.
Similarly, we can also get a new patch $N_2^5=N_0^5\cup L_1\cup L_2$ which also has the boundary same with $N_0^5$ if
each face of the layer $L_2$ is a hexagon, where $L_2$ is the layer around $N_1^5$.
Thus, we can do the this operation repeatedly until the $m$th step such that the layer $L_{m+1}$ around $N_m^5$ contains a quadrilateral, where $N_m^5=N_0^5\cup L_1\cup \cdots\cup L_m$. Denote the boundary of $N_m^5$ the cycle $a_8'y_6'a_6'y_2'a_3'y_3'a_7'y_7'a_9'y_9'a_{10}'y_8'a_8'$ and $g_i^5$($1\leqslant i\leqslant6$) the faces of the layer $L_{m+1}$ in the clockwise direction, where $g_1^5$ is the face along the path $y_6'a_6'y_2'$, see Fig. \ref{figH_m^5}. 

 \begin{figure}[h]
    \centering
     \subfigure[
   $F-G_1=H_{28}$]{
   \label{figB93}
   \includegraphics[height=29mm]{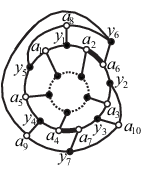}}
     \subfigure[the patch $N_m^5$]{
   \label{figH_m^5}
   \includegraphics[height=28mm]{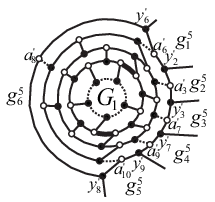}}
    \subfigure[$F-G_1=H_{54}^m$]{
    \label{figB94}
    \includegraphics[height=28mm]{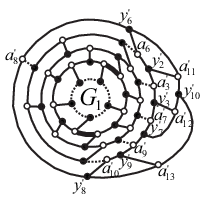}}
     \subfigure[$F-G_1=H_{55}^m$]{
   \label{figB95}
   \includegraphics[height=28mm]{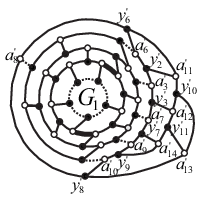}}
    \subfigure[$F-G_1=H_{56}^m$]{
    \label{figB96}
    \includegraphics[height=28mm]{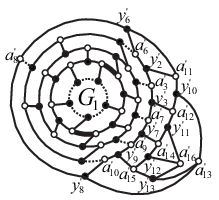}}
    \subfigure[$F-G_1=H_{57}^m$]{
   \label{figB97}
   \includegraphics[height=27mm]{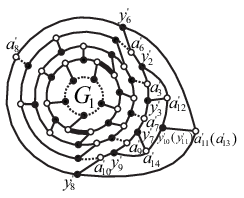}}
    \subfigure[$F-G_1=H_{58}^m$]{
    \label{figB98}
    \includegraphics[height=26mm]{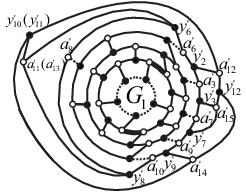}}
     \subfigure[$F-G_1=H_{59}^m$]{
   \label{figB99}
   \includegraphics[height=26mm]{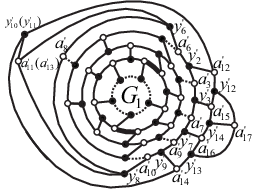}}
    \subfigure[$F-G_1=H_{60}^m$]{
    \label{figB100}
    \includegraphics[height=27mm]{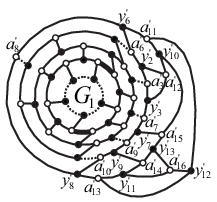}}
     \subfigure[$F-G_1=H_{61}^m$]{
   \label{figB101}
   \includegraphics[height=27mm]{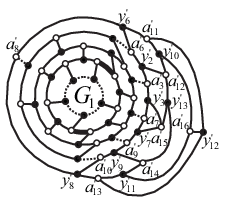}}
 \caption{Illustration for Subcase 2.1.2.}
   \end{figure}

If $g_1^5$ or $g_5^5$ is a quadrilateral, without loss of generality, let $g_1^5$ being a quadrilateral, then $y_6'$ and $y_2'$ have a common neighbor $a_{11}'$.  Let $y_{10}'$ be the third neighbor of $a_{11}'$.
Then, $y_{10}'\notin\{y_3', y_7',y_8',y_9'\}$ by Lemma \ref{lem1}.
Thus, $y_{10}'$ and $y_3'$ have a common neighbor $a_{12}'$ while $y_{10}'$ and $y_8'$ have a common neighbor $a_{13}'$. Clearly, $a_{13}'\neq a_{12}'$. If $a_{12}'y_7'\in E(F)$, then $a_{13}'y_9'\in E(F)$.  Thus, $F-G_1=H_{54}^m$; see Fig. \ref{figB94}. If $a_{12}'y_7'\notin E(F)$, then there is a path $a_{12}'y_{11}'a_{14}'y_7'$.  If $a_{14}'y_9'\in E(F)$, then $a_{13}'y_{11}'\in E(F)$. Hence, $F-G_1=H_{55}^m$; see Fig. \ref{figB95}.
If $a_{14}'y_9'\notin E(F)$, then there is a path $a_{14}'y_{12}'a_{15}'y_9'$. Further, $a_{15}'$ and $a_{13}'$ have a common neighbor $y_{13}'$. Thus, $F-G_1=H_{56}^m$ by Lemma \ref{lem1}; see Fig. \ref{figB96}. If both of $g_1^5$ and $g_5^5$ are hexagons, then there are two paths $ y_6'a_{11}'y_{10}'a_{12}'y_2'$ and $y_8'a_{13}'y_{11}'a_{14}'y_9'$ respectively.

If $g_6^5$ is a quadrilateral, then $a_{11}'=a_{13}'$ and $y_{11}'=y_{10}'$.
If $a_{12}'y_3'\in E(F)$, then $a_{14}'y_7'\in E(F)$, i.e., $F-G_1=H_{57}^m$; see Fig. \ref{figB97}.
If $a_{12}'y_3'\notin E(F)$, then there is a path $a_{12}'y_{12}'a_{15}'y_3'$.
If $a_{15}'y_7'\in E(F)$, then $a_{14}'y_{12}'\in E(F)$. Thus, $F-G_1=H_{58}^m$; see Fig. \ref{figB98}.
If $a_{15}'y_7'\notin E(F)$, then there is a path $a_{15}'y_{14}'a_{16}'y_7'$. Further, $a_{14}'$ and $a_{16}'$ have a common neighbour $y_{13}'$.
 Thus, $F-G_1=H_{59}^m$ by Lemma \ref{lem1}; see Fig. \ref{figB99}. If $g_6^5$ is a hexagon, then
 $a_{11}'$ and $a_{13}'$ have a common neighbor $y_{12}'$.

 If $g_2^5$ or $g_4^5$ is a quadrilateral, without loss of generality, let $g_2^5$ is a quadrilateral, i.e., $a_{12}'y_3'\in E(F)$, then $y_{10}'$ and $y_7'$ have a common neighbor $a_{15}'$.
Further, $a_{14}'$ and $a_{15}'$ have a common neighbor $y_{13}'$. Thus, $F-G_1=H_{60}^m$ by Lemma \ref{lem1}; see Fig. \ref{figB100}.

If both of $g_2^5$ and $g_4^5$ are hexagons, then $g_3^5$ must be a quadrilateral as the layer $L_{m+1}$ contains a quadrilateral, i.e., $y_3'$ and $y_7'$ have a common neighbor $a_{15}'$. Thus, $a_{12}'$ and $a_{15}'$ have a common neighbor $y_{13}'$. Further, $a_{14}'y_{13}'\in E(F)$. Thus, $F-G_1=H_{61}^m$ by Lemma \ref{lem1}; see Fig. \ref{figB101}.

\textbf{Subcase 2.2.} Edge $e_2\in E(f_4)$. Then, $f_i\cap (F-G_1)$ is a path 
$P_3$, say $a_iy_ia_{i+1}$, where $i\in\{3, 5\}$ and  $a_6=a_1$.

If $a_4$ is incident with $e_2$, then let $f_4\cap (F-G_1)=a_4a_7y_4a_5$. Denote $a_8,a_9, y_6$ and $y_7$  the third neighbors of $y_1,y_3,a_6$ and $a_7$ respectively. By Lemmas \ref{lem1} and \ref{lem2}, 
any two vertices of $\{a_6,a_7,a_8,a_9,y_1,\ldots, y_7\}$ are different. Thus, $\{a_8y_6,a_9y_7\}\subseteq E(F)$ by Observation \ref{claim2}(ii). Then, we get a new patch $N_0^6$. Further, the outer face of $N_0^6$ is of size 14 and the degrees of the vertices on its boundary form a degree-array [2,2,3,2,3,3,2,2,3,2,3,2,3,3] with two white 2-degree vertices and five black 2-degree vertices which read from $a_8$ in the clockwise direction.
By Observation \ref{claim2}(i),
every face of $F$ but not $N_0^6$ is either a quadrilateral or a hexagon. Denote $L_1$ the layer around the patch $N_0^6$.
\begin{figure}
   \subfigure[the patch $N_m^6$]{
   \label{H_0^2}
   \includegraphics[height=30mm]{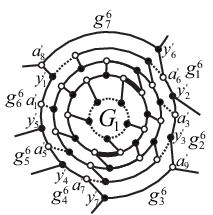}}
    \subfigure[$F-G_1=H_{62}^m$]{
    \label{figB102}
    \includegraphics[height=28mm]{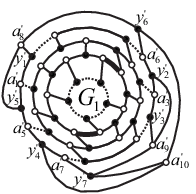}}
    \subfigure[$F-G_1=H_{63}^m$]{
    \label{figB103}
    \includegraphics[height=28mm]{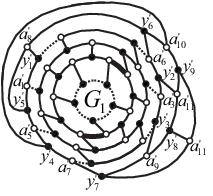}}
     \subfigure[$F-G_1=H_{64}^m$]{
   \label{figB104}
   \includegraphics[height=28mm]{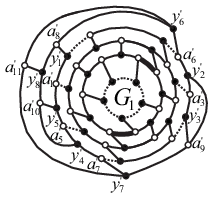}}
    \subfigure[$F-G_1=H_{65}^m$]{
    \label{figB105}
    \includegraphics[height=28mm]{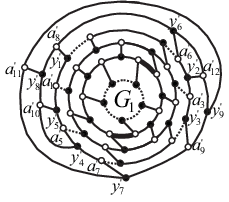}}
     \subfigure[$F-G_1=H_{66}^m$]{
   \label{figB106}
   \includegraphics[height=28mm]{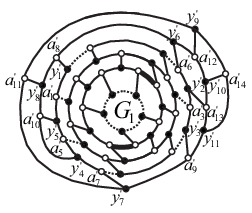}}
    \subfigure[$F-G_1=H_{67}^m$]{
    \label{figB107}
    \includegraphics[height=28mm]{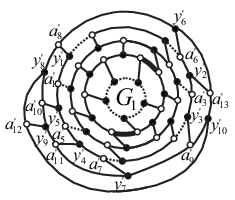}}
     \subfigure[$F-G_1=H_{68}^m$]{
   \label{figB108}
   \includegraphics[height=28mm]{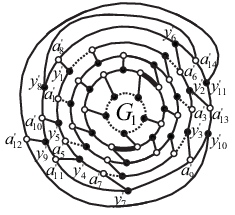}}
    \subfigure[$F-G_1=H_{69}^m$]{
    \label{figB109}
    \includegraphics[height=28mm]{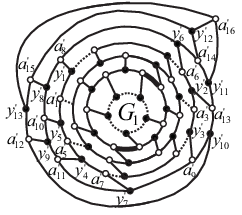}}
     \subfigure[$F-G_1=H_{70}^m$]{
   \label{figB110}
   \includegraphics[height=28mm]{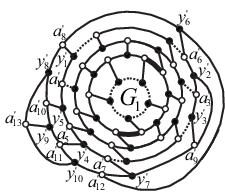}}
    \subfigure[$F-G_1=H_{71}^m$]{
    \label{figB111}
    \includegraphics[height=28mm]{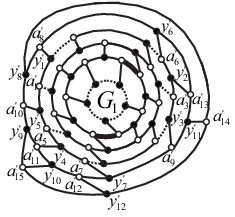}}
     \subfigure[$F-G_1=H_{72}^m$]{
   \label{figB112}
   \includegraphics[height=28mm]{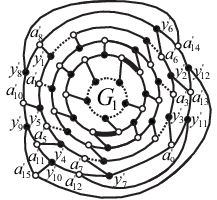}}
    \subfigure[$F-G_1=H_{73}^m$]{
    \label{figB113}
    \includegraphics[height=28mm]{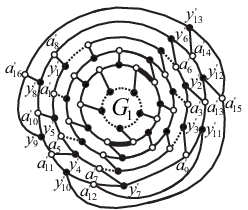}}
     ~~~~
     \subfigure[$F-G_1=H_{74}^m$]{
   \label{figB114}
   \includegraphics[height=28mm]{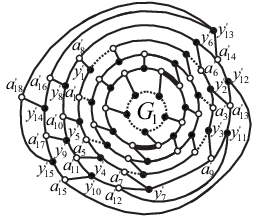}}
        ~~~~~
    \subfigure[$F-G_1=H_{75}^m$]{
    \label{figB115}
    \includegraphics[height=28mm]{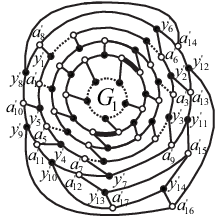}}
        ~~~~~~~
     \subfigure[$H_{22}\subseteq F-G_1$]{
   \label{figB116}
   \includegraphics[height=28mm]{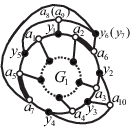}}
    \caption{Illustration for  Subcase 2.2.}
\end{figure}
By Lemma \ref{lem1}, we can get a new patch $N_1^6=N_0^6\cup L_1$ which has the boundary same with $N_0^6$ if each face of $L_1$ is a hexagon.
Denote $L_2$ the layer around $N_1^6$.
Similarly, we can also get a new patch $N_2^6=N_0^6\cup L_1\cup L_2$ which has the boundary same with $N_0^6$ if each face of $L_2$ is a hexagon.
 Thus, we can do the above operation repeatedly until the $m$th step such that the layer $L_{m+1}$ around $N_m^6$ contains a quadrilateral, where $N_m^6=N_0^6\cup L_1\cup\cdots \cup L_m$. Denote the boundary of $N_m^6$ the cycle $a_8'y_6'a_6'y_2'a_3'y_3'a_9'y_7'a_7'y_4'a_5'y_5'a_1'y_1'a_8'$ and $g_i^6$($1\leqslant i\leqslant7$)  the faces of $L_{m+1}$ in the clockwise direction, where $g_1^6$ is the face along the path $y_6'a_6'y_2'$, see Fig. \ref{H_0^2}.

If $g_6^6$ is a quadrilateral, i.e., $a_8'y_5'\in E(F)$, then $y_4'$ and $y_6'$ have a common neighbor $a_{10}'$. If $y_2'a_9'\in E(F)$, then $a_{10}'y_7'\in E(F)$. Thus, $F-G_1=H_{62}^m$; see Fig. \ref{figB102}. If $y_2'a_9'\notin E(F)$, then there is a path $y_2'a_{11}'y_8'a_9'$. Further, $a_{10}'$ and $a_{11}'$ have a common neighbor $y_9'$. Thus,  $F-G_1=H_{63}^m$ by Lemma \ref{lem1}; see Fig. \ref{figB103}.
If $g_6^6$ is a hexagon, then there is a path $a_8'y_8'a_{10}'y_5'$ and $a_{10}'\neq a_9'$ by Lemma \ref{lem1}.

If $g_5^6$ is a quadrilateral, i.e., $a_{10}'y_4'\in E(F)$, then $y_8'$ and $y_7'$ have a common neighbor $a_{11}'$. If $a_{11}'y_6'\in E(F)$, then $y_2'a_9'\in E(F)$. Thus, $F-G_1=H_{64}^m$; see Fig. \ref{figB104}.
If $a_{11}'y_6'\notin E(F)$, then there is a path $a_{11}'y_9'a_{12}'y_6'$. If $a_{12}'y_2'\in E(F)$, then $a_9'y_9'\in E(F)$. That is, $F-G_1=H_{65}^m$; see Fig. \ref{figB105}.   If $a_{12}'y_2'\notin E(F)$, then there is a path $a_{12}'y_{10}'a_{13}'y_2'$. Further, $a_{13}'$ and $a_9'$ have a common neighbor $y_{11}'$. Thus, $F-G_1=H_{66}^m$ by Lemma \ref{lem1}; see Fig. \ref{figB106}.
If $g_5^6$ is a hexagon, then there is a path $a_{10}'y_9'a_{11}'y_4'$.

If $g_4^6$ is a quadrilateral, then $a_{11}'y_7'\in E(F)$. Further, $y_9'a_9'\notin E(F)$ since otherwise the face along the path $y_8'a_{10}'y_9'a_9'y_3'a_3'y_2'$ is of size at least 8, a contradiction. Hence, there is a path $y_9'a_{12}'y_{10}'a_9'$. Further, $y_{10}'$ and $y_2'$ have a common neighbor $a_{13}'$. If $a_{13}'y_6'\in E(F)$, then $a_{12}'y_8'\in E(F)$. Thus, $F-G_1=H_{67}^m$; see Fig. \ref{figB107}.  If $a_{13}'y_6'\notin E(F)$, then there is a path $a_{13}'y_{11}'a_{14}'y_6'$. If $y_8'a_{14}'\in E(F)$, then $y_{11}'a_{12}'\in E(F)$. Thus, $F-G_1=H_{68}^m$; see Fig. \ref{figB108}. If $y_8'a_{14}'\notin E(F)$, then there is a path $y_8'a_{15}'y_{12}'a_{14}'$. Further, $a_{12}'$ and $a_{15}'$ have a common neighbor $y_{13}'$. Thus, $F-G_1=H_{69}^m$ by Lemma \ref{lem1}; see Fig. \ref{figB109}.
If $g_4^6$ is a hexagon, then there is a path $a_{11}'y_{10}'a_{12}'y_7'$.

If $g_2^6$ is a quadrilateral, then $y_2'a_9'\in E(F)$. Further, $a_{12}'y_6'\in E(F)$. Thus, $F-G_1=H_{70}^m$ by Lemma \ref{lem1}; see Fig. \ref{figB110}.
If $g_2^6$ is a hexagon, then there is a path $y_2'a_{13}'y_{11}'a_9'$.

If $g_1^6$ is a quadrilateral, i.e., $a_{13}'y_6'\in E(F)$, then $y_{11}'$ and $y_8'$ have a common neighbor $a_{14}'$. Further, $a_{14}'$ and $a_{12}'$ have a common neighbor $y_{12}'$. Thus, $F-G_1=H_{71}^m$ by Lemma \ref{lem1}; see Fig. \ref{figB111}.
If $g_1^6$ is a hexagon, then there is a path $a_{13}'y_{12}'a_{14}'y_6'$.

If $g_3^6$ is a quadrilateral, then $a_{12}'y_{11}'\in E(F)$. Further, $y_{10}'$ and $y_{12}'$ have a common neighbor $a_{15}'$. If $y_8'a_{14}'\in E(F)$, then $y_9'a_{15}'\in E(F)$. Thus, $F-G_1=H_{72}^m$; see Fig. \ref{figB112}.
If $a_{14}'y_8'\notin E(F)$, then there is a path $a_{14}'y_{13}'a_{16}'y_8'$. If $a_{16}'y_9'\in E(F)$, then $a_{15}'y_{13}'\in E(F)$. That is, $F-G_1=H_{73}^m$; see Fig. \ref{figB113}.
If $a_{16}'y_9'\notin E(F)$, then there is a path $a_{16}'y_{14}'a_{17}'y_9'$. Further, $a_{17}'$ and $a_{15}'$ have a common neighbor $y_{15}'$. Thus, $F-G_1=H_{74}^m$  by Lemma \ref{lem1}; see Fig. \ref{figB114}.
 If $g_3^6$ is a hexagon, then there is a path $a_{12}'y_{13}'a_{15}'y_{11}'$.

Thus, $g_7^6$ is a quadrilateral as $L_{m+1}$ contains a  quadrilateral, i.e., $a_{14}'y_8'\in E(F)$. Thus, $y_{12}'$ and $y_9'$ have a common neighbor $a_{16}'$, and $a_{15}'$ and $a_{16}'$ have a common neighbor $y_{14}'$. Thus, $F-G_1=H_{75}^m$ by Lemma \ref{lem1}; see Fig. \ref{figB115}.

If $a_5$ is incident with $e_2$, then let $f_4\cap (F-G_1)=a_4y_4a_7a_5$. Let $a_8, a_9,y_6$ and $y_7$ be the third neighbors of $y_1,y_5,a_6$ and $a_7$ respectively. Clearly, $a_8\notin\{a_6, a_7\},a_9\notin\{a_6, a_7\}, y_6\notin \{y_1,\ldots, y_5\}$ and $y_7\notin \{y_1,\ldots, y_5\}$. Thus, $e_1$ and $e_2$ do not lie on the same face of $F$.  By Observation \ref{claim2}(ii), $\{a_8y_6,a_9y_7\}\subseteq E(F)$. 
If $a_9=a_8$, then $y_7=y_6$. Thus, $H_{22}\subseteq F-G_1$  by Lemma \ref{lem1}; see Fig. \ref{figB116}. If $a_9\neq a_8$, then $y_7\neq y_6$. 
Further, $a_8$ and $a_9$ have a common neighbor $y_8$. Thus, we get a patch $N_0^7$ which has the boundary same with $N_0^5$ of  Subcase 3.2.1.2.
Thus, similar as the analysis of $F-N_0^5$, we can get $F-G_1\in \{H_{76}^m, \ldots, H_{82}^m\}$.

\textbf{Subcase 2.3.} Edge $e_2\in E(f_5)$.  Then, let $f_i\cap (F-G_1)=a_iy_ia_{i+1}$ for  $i=3, 4$.

If $a_5$ is incident with $e_2$, then it is similar to the case that $a_4$ is incident with $e_2$ of Subcase 3.2.2. That is, $F-G_1\in \{H_{62}^m, \ldots, H_{75}^m\}$.

If $a_1$ is incident with $e_2$, then we can assume $f_i\cap (F-G_1)=a_5y_5a_7a_1$. Let $a_8, y_6$ and $y_7$ be the third neighbors of $y_1, a_6$ and $a_7$ respectively. Clearly, $a_8\notin\{a_6,a_7\}$ and any two vertices of $\{y_1,\ldots, y_7\}$ are different. Further, $\{a_8y_6,a_8y_7\}\subseteq E(F)$. Then, we get a new patch $N_0^8$ 
which has the boundary same with $N_0^5$. Thus, similar as the analysis of $F-N_0^5$, we can get $F-G_1\in \{H_{83}^m, \ldots, H_{90}^m\}$.

\textbf{Subcase 2.4.} Edge $e_2\notin \cup_{i=3}^5 E(f_i)$. Then, for $3\leqslant i\leqslant5$, let $f_i\cap (F-G_1)=a_iy_ia_{i+1}$, where $a_6=a_1$. Let $a_7$ and $y_6$ be the third neighbors of $y_1$ and $a_6$ respectively. Then $a_7\neq a_6$ and $y_6\notin\{y_1,\ldots,y_5\}$ by Lemmas \ref{lem1} and \ref{lem2}.

\textbf{Subcase 2.4.1.} There is a face  of $F$   containing the edge set $E_0$. Then,  $a_7$ and $y_6$ have a common neighbor $a_8$ by Observation \ref{claim2}(iii).
 \begin{figure}[h]
    \centering
   \subfigure[$F-G_1=H_{29}$]{
   \label{figB37}
   \includegraphics[height=30mm]{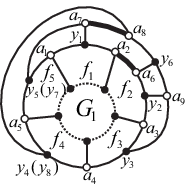}}
    \subfigure[the patch $N_m^9$]{
    \label{figH^m}
    \includegraphics[height=30mm]{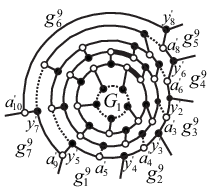}}
    \subfigure[$F-G_1=H_{91}^m$]{
    \label{figB38}
    \includegraphics[height=30mm]{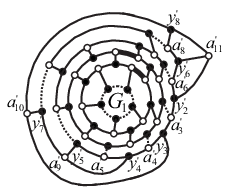}}
     \subfigure[$F-G_1=H_{92}^m$]{
    \label{figB39}
    \includegraphics[height=30mm]{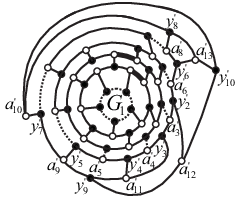}}
     \subfigure[$F-G_1=H_{93}^m$]{
    \label{figB40}
    \includegraphics[height=30mm]{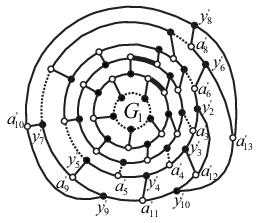}}
     \subfigure[$F-G_1=H_{94}^m$]{
    \label{figB41}
    \includegraphics[height=30mm]{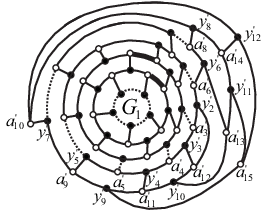}}
     \subfigure[$F-G_1=H_{95}^m$]{
    \label{figB42}
    \includegraphics[height=30mm]{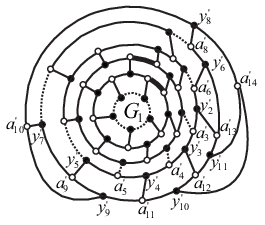}}
     \subfigure[$F-G_1=H_{96}^m$]{
    \label{figB43}
    \includegraphics[height=30mm]{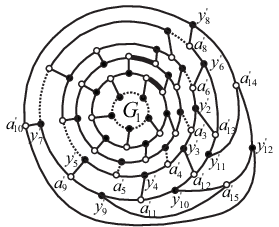}}
     \subfigure[$F-G_1=H_{97}^m$]{
    \label{figB44}
    \includegraphics[height=30mm]{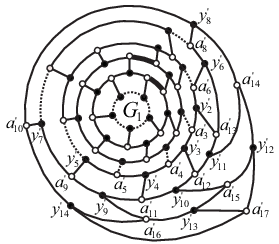}}
     \subfigure[$F-G_1=H_{98}^m$]{
    \label{figB46}
    \includegraphics[height=30mm]{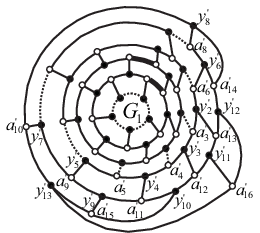}}
     \subfigure[$F-G_1=H_{99}^m$]{
    \label{figB47}
    \includegraphics[height=30mm]{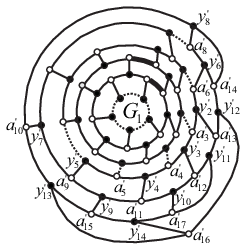}}
     \subfigure[$F-G_1=H_{100}^m$]{
    \label{figB48}
    \includegraphics[height=30mm]{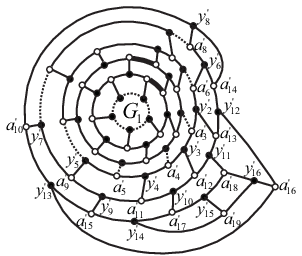}}
     \subfigure[$F-G_1=H_{101}^m$]{
    \label{figB49}
    \includegraphics[height=30mm]{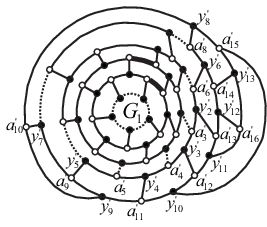}}
     \subfigure[$F-G_1=H_{102}^m$]{
    \label{figB51}
    \includegraphics[height=30mm]{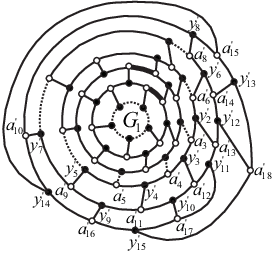}}
     \subfigure[$F-G_1=H_{103}^m$]{
    \label{figB52}
    \includegraphics[height=30mm]{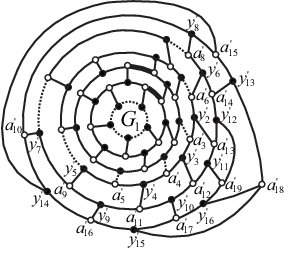}}
     \subfigure[$F-G_1=H_{104}^m$]{
    \label{figB53}
    \includegraphics[height=30mm]{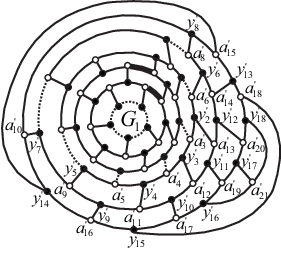}}
    \caption{Illustration for Subcase 2.4.1.}
\end{figure}

 Let  $y_i$ be the third neighbor of $a_i$, $i=7,8$.
By Lemmas \ref{lem1} and \ref{lem2}, $y_7\notin\{y_2,y_3,y_4,y_6,y_8\}$ and $y_8\notin\{y_2,y_3,y_5\}$. If $y_7=y_5$, then $y_8=y_4$. Thus, $F-G_1=H_{29}$ by Lemma \ref{lem1}; see Fig. \ref{figB37}. If $y_7\neq y_5$, then $y_8\neq y_4$ and there is a vertex $a_9$ adjacent to $y_7$ and $y_5$. Further,  $y_7$ and $y_8$ have a common neighbor $a_{10}$ by Observation \ref{claim2}(ii) and $a_{10}\neq a_9$. Thus, we get a new patch $N_0^9$
and the outer face $N_0^9$ is of length 14. Further, the degrees of the vertices on its boundary form a degree-array [2,2,3,2,3,2,3,2,3,2,3,3,2,3] with two white 2-degree vertices and five black 2-degree vertices which is read form vertex $a_{10}$ in the clockwise direction.
Then, every face of $F$ but not $N_0^9$ is either a quadrilateral or a hexagon by Observation \ref{claim2}(i).
Denote $L_1$ the layer around $N_0^9$.
Then, by Lemma \ref{lem1}, we can get a new patch $N_1^9=N_0^9\cup L_1$ which  has the boundary same with $N_0^9$ if  each face of $L_1$ is a hexagon.
 Similarly, we can further get a new patch $N_2^9=N_0^9\cup L_1\cup L_2$ which also has the boundary same with $N_0^9$ if each face of $L_2$ is a hexagon, where $L_2$ is the layer around $N_1^9$.
Do this operation until $m$th step such that  the layer $L_{m+1}$ around $N_m^9$ contains a quadrilateral, where $N_m^9=N_0^9\cup L_1\cup \cdots \cup L_m$.
Denote the boundary of $N_m^9$ the cycle $a_{10}'y_8'a_8'y_6'a_6'y_2'a_3'y_3'a_4'y_4'a_5'y_5'a_9'y_7'a_{10}'$ and $g_i^9$($1\leqslant i\leqslant7$) the faces of the layer $L_{m+1}$ in the counterclockwise direction, where $g_1^9$ is the face along the path $a_9'y_5'a_5'y_4'$, see Fig. \ref{figH^m}.

 If $g_1^9$ is a quadrilateral, then $a_9'y_4'\in E(F)$ and $a_{10}'y_3'\in E(F)$. Thus,
$F-G_1=H_{91}^m$ by Lemma \ref{lem1}; see Fig. \ref{figB38}.
If $g_1^9$ is a hexagon, then there is a path $a_9'y_9'a_{11}'y_4'$.

If $g_2^9$ is a quadrilateral, then $a_{11}'y_3'\in E(F)$ while $y_9'$ and $y_2'$ have a common neighbor $a_{12}'$. Further, $a_{10}'$ and $a_{12}'$  have a common neighbor $y_{10}'$. Thus, $F-G_1=H_{92}^m$ by Lemma \ref{lem1}; see Fig. \ref{figB39}.
If $g_2^9$ is a hexagon, then there is a path $a_{11}'y_{10}'a_{12}'y_3'$.

If $g_3^9$ is a quadrilateral, then $a_{12}'y_2'\in E(F)$ while $y_{10}'$ and $y_6'$ have a common neighbor $a_{13}'$.
If $y_8'a_{13}'\in E(F)$, then $a_{10}'y_9'\in E(F)$. Thus, $F-G_1=H_{93}^m$; see Fig. \ref{figB40}. If $y_8'a_{13}'\notin E(F)$, then there is a path $y_8'a_{14}'y_{11}'a_{13}'$. Thus, $y_9'$ and $y_{11}'$ have a common neighbor $a_{15}'$. Thus, $F-G_1=H_{94}^m$ by Lemma \ref{lem1}; see Fig. \ref{figB41}.
If $g_3^9$ is a hexagon, then there is a path $a_{12}'y_{11}'a_{13}'y_2'$.

If $g_4^9$ is a quadrilateral, then
$a_{13}'y_6'\in E(F)$ while $y_8'$ and $y_{11}'$ have a common neighbor $a_{14}'$. If $a_{14}'y_{10}'\in E(F)$, then $a_{10}'y_9'\in E(F)$. That is, $F-G_1=H_{95}^m$; see Fig. \ref{figB42}. If $a_{14}'y_{10}'\notin E(F)$, then there is a path $a_{14}'y_{12}'a_{15}'y_{10}'$. If $y_9'a_{15}'\in E(F)$, then $a_{10}'y_{12}'\in E(F)$. Thus, $F-G_1=H_{96}^m$; see Fig. \ref{figB43}. If $y_9'a_{15}'\notin E(F)$, then there is a path $y_9'a_{16}'y_{13}'a_{15}'$. Further, $a_{10}'$ and $a_{16}'$ have a common neighbor $y_{14}'$. Thus,  $F-G_1=H_{97}^m$ by Lemma \ref{lem1}; see Fig. \ref{figB44}.
If $g_4^9$ is a hexagon, then there is a path $y_6'a_{14}'y_{12}'a_{13}'$.

If $g_5^9$ is a quadrilateral, then $y_8'a_{14}'\in E(F)$. Further, $a_{10}'y_9'\notin E(F)$ since otherwise the face along the path $y_{10}'a_{11}'y_9'a_{10}'y_8'a_{14}'y_{12}'$ is of size at least 8, a contradiction. Thus, there is a path $a_{10}'y_{13}'a_{15}'y_9'$.
 Further, $y_{13}'$ and $y_{12}'$ have a common neighbor $a_{16}'$.
 If $a_{15}'y_{10}'\in E(F)$, then $a_{16}'y_{11}'\in E(F)$. That is, $F-G_1=H_{98}^m$; see Fig. \ref{figB46}.  If $a_{15}'y_{10}'\notin E(F)$, then there is a path $a_{15}'y_{14}'a_{17}'y_{10}'$. 
 If $y_{11}'a_{17}'\in E(F)$, then $a_{16}'y_{14}'\in E(F)$. Thus, $F-G_1=H_{99}^m$; see Fig. \ref{figB47}. If $y_{11}'a_{17}'\notin E(F)$, then there is a path $y_{11}'a_{18}'y_{15}'a_{17}'$ by Lemma \ref{lem1}. Thus, $a_{18}'$ and $a_{16}'$ have a common neighbor $y_{16}'$. By Lemma \ref{lem1}, $F-G_1=H_{100}^m$; see Fig. \ref{figB48}.
If $g_5^9$ is a hexagon, then there is a path $y_8'a_{15}'y_{13}'a_{14}'$.

If  $g_7^9$ is a quadrilateral,
then $a_{10}'y_9'\in E(F)$ and $a_{15}'y_{10}'\in E(F)$. Thus, $F-G_1=H_{101}^m$ by Lemma \ref{lem1}; see Fig. \ref{figB49}.
If $g_7^9$ is a hexagon, then there is a path $a_{10}'y_{14}'a_{16}'y_9'$.

Since the layer $L_{m+1}$ contains a quadrilateral, $g_6^9$ is a quadrilateral, i.e., $a_{15}'y_{14}'\in E(F)$.
If $y_{10}'a_{16}'\in E(F)$, then 
there is a face of $F$ with size at least 8, a contradiction.
 Hence, there is a path $y_{10}'a_{17}'y_{15}'a_{16}'$.  Further,  $y_{13}'$ and $y_{15}'$ have a common neighbor $a_{18}'$.
If $y_{11}'a_{17}'\in E(F)$, then $y_{12}'a_{18}'\in E(F)$. Thus, $F-G_1=H_{102}^m$; see Fig. \ref{figB51}. If $y_{11}'a_{17}'\notin E(F)$,  then there is a path $y_{11}'a_{19}'y_{16}'a_{17}'$. If $a_{19}'y_{12}'\in E(F)$, then $a_{18}'y_{16}'\in E(F)$. Thus, $F-G_1=H_{103}^m$; see Fig. \ref{figB52}. If $a_{19}'y_{12}'\notin E(F)$, then there is a path $a_{19}'y_{17}'a_{20}'y_{12}'$. Further, $a_{18}'$ and $a_{20}'$ have a common neighbor $y_{18}'$. Thus,  $F-G_1=H_{104}^m$ by Lemma \ref{lem1}; see Fig. \ref{figB53}.
 \begin{figure}[h]
    \centering
    \subfigure[the patch $N_m^0$]{
    \label{figH_m}
    \includegraphics[height=27mm]{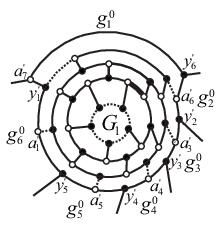}}
    \subfigure[$H_{22}\subseteq F-G_1$]{
    \label{figB54}
    \includegraphics[height=25mm]{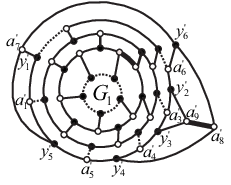}}
     \subfigure[$F-G_1=H_{105}^m$]{
    \label{figB55}
    \includegraphics[height=26mm]{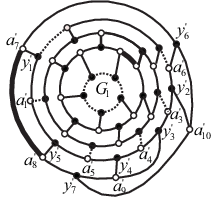}}
     \subfigure[$F-G_1=H_{106}^m$]{
    \label{figB56}
    \includegraphics[height=26mm]{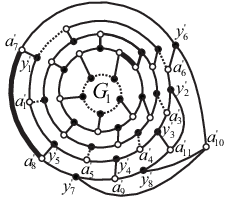}}
     \subfigure[$F-G_1=H_{107}^m$]{
    \label{figB57}
    \includegraphics[height=25mm]{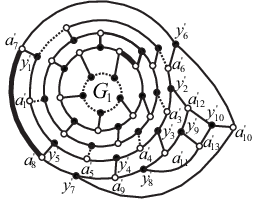}}
     \subfigure[
     $H_{22}\subseteq F-G_1$]{
    \label{figB58}
    \includegraphics[height=26mm]{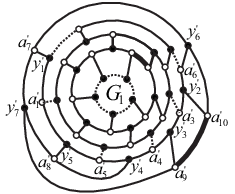}}
     \subfigure[$F-G_1=H_{108}^m$]{
    \label{figB59}
    \includegraphics[height=27mm]{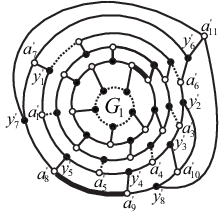}}
     \subfigure[$F-G_1=H_{109}^m$]{
    \label{figB60}
    \includegraphics[height=26mm]{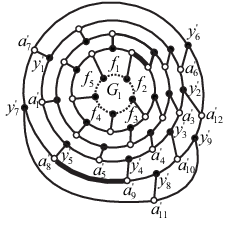}}
     \subfigure[$F-G_1=H_{110}^m$]{
    \label{figB61}
    \includegraphics[height=26mm]{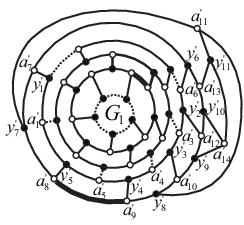}}
     \subfigure[
    $H_{22}\subseteq F-G_1$]{
    \label{figB62}
    \includegraphics[height=26mm]{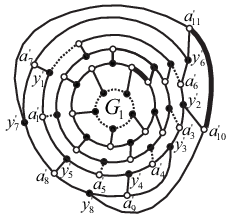}}
     \subfigure[$F-G_1=H_{111}^m$]{
    \label{figB63}
    \includegraphics[height=26mm]{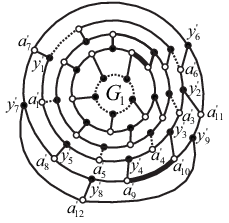}}
     \subfigure[$F-G_1=H_{112}^m$]{
    \label{figB64}
    \includegraphics[height=26mm]{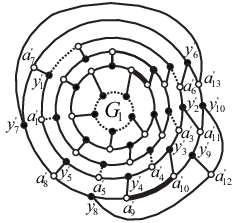}}
     \subfigure[$F-G_1=H_{113}^m$]{
    \label{figB65}
    \includegraphics[height=26mm]{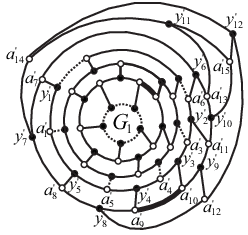}}
     \subfigure[
     $H_{22}\subseteq F-G_1$]{
    \label{figB66}
    \includegraphics[height=27mm]{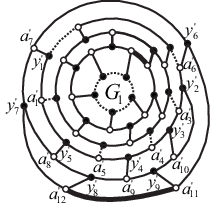}}
     \subfigure[$F-G_1=H_{114}^m$]{
    \label{figB67}
    \includegraphics[height=26mm]{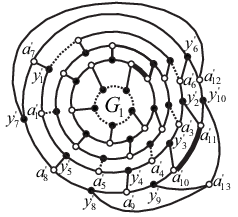}}

   \caption{Illustration for Subcase 2.4.2.}

\end{figure}

\textbf{Subcase 2.4.2.} There is no face of $F$ containing the edge set $E_0$. Then $a_7y_6\in E(F)$ by Observation \ref{claim2}(ii).
Thus, we get a new patch $N_0^0$.
Further, the outer face of $N_0^0$ is of length 12 and the degrees of vertices on its boundary form a degree-array [2,2,3,2,3,2,3,2,3,2,3,3] with one white 2-degree vertex and five black 2-degree vertices  which is read from $a_7$ in the clockwise direction. Denote $L_1$ the layer around the patch $N_0^0$.
 If every face of $L_1$ is neither  a quadrilateral nor contains edge $e_2$, then $L_1$ exactly consists of six hexagons by Observation \ref{claim2}(i). Further, the new patch $N_1^0=N_0^0\cup L_1$ has the boundary same as $N_0^0$ if each face of $L_1$ is a hexagon. Denote $L_2$ the layer around $N_1^0$. Then,
 we can also get a new patch $N_2^0=N_0^0\cup L_1\cup L_2$ which also has the boundary same with $N_0^0$ if each face of $L_2$ is a hexagon.
Thus, we can do the above operation repeatedly until $m$th step such that there is a face of the layer $L_{m+1}$ around $N_m^0$ being a quadrilateral or containing edge $e_2$, where $N_m^0=N_0^0\cup L_1\cup\cdots\cup L_m$. Let the boundary of $N_m^0$ be the cycle $a_7'y_6'a_6'y_2'a_3'y_3'a_4'y_4'a_5'y_5'a_1'y_1'a_7'$ and $g_i^0$($1\leqslant i\leqslant6$) the faces of $L_{m+1}$ in the clockwise direction, where $g_1^0$ is the face containing edge $a_7'y_6'$, see Fig. \ref{figH_m}.

If $g_6^0$ is a quadrilateral, then $a_7'y_5'\in E(F)$. Further, $y_4'$ and $y_6'$ have a common neighbor $a_8'$.
Thus, $a_8', y_2'$ and $y_3'$ have a common neighbor $a_9'$ by Lemma \ref{lem1} and the fact that $F_0'$ is bipartite, i.e., $e_2=a_8'a_9'$ and $H_{22}\subseteq F-G_1$; see Fig. \ref{figB54}.
If $e_2\in E(g_6^0)$, then there is a path $a_7'a_8'y_5'$ by Observation \ref{claim2}(ii). Further, $y_4'a_8'\notin E(F)$ since otherwise the face along the path $y_3'a_4'y_4'a_8'a_7'y_6'$ is of size 7 by Lemma \ref{lem1},
a contradiction. Hence, there is a path $y_4'a_9'y_7'a_8'$. Further, $y_7'$ and $y_6'$ have a common neighbor $a_{10}'$ by Observation \ref{claim2}(ii). If
$y_3'a_9'\in E(F)$, then $a_{10}'y_2'\in E(F)$. Thus, $F-G_1=H_{105}^m$; see Fig. \ref{figB55}. If
$y_3'a_9'\notin E(F)$, then there is a path $y_3'a_{11}'y_8'a_9'$.
If $y_2'a_{11}'\in E(F)$, then $y_8'a_{10}'\in E(F)$. Thus, $F-G_1=H_{106}^m$; see Fig. \ref{figB56}. If $y_2'a_{11}'\notin E(F)$, then there is a path $y_2'a_{12}'y_9'a_{11}'$. Further, $a_{10}'$ and $a_{12}'$ have a common neighbor $y_{10}'$.
Thus, $F-G_1=H_{107}^m$ by Lemma \ref{lem1}; see Fig. \ref{figB57}.
If $g_6^0$ is a hexagon, then there is a path $a_7'y_7'a_8'y_5'$.

If $g_5^0$ is a quadrilateral, then $a_8'y_4'\in E(F)$. Further, $y_7'$ and $y_3'$ have a common neighbor $a_9'$. Thus, $H_{22}\subseteq F-G_1$ 
 by Lemma \ref{lem1}; see Fig. \ref{figB58}. If $e_2\in E(g_5^0)$, then there is a path $a_8'a_9'y_4'$ by Observation \ref{claim2}(ii).
By the restriction on the faces of $F$, $y_3'a_9'\notin E(F)$.
Thus,  there is a path $y_3'a_{10}'y_8'a_9'$. Further, $y_7'$ and $y_8'$ have a common neighbor $a_{11}'$ by Observation \ref{claim2}(ii). If $a_{10}'y_2'\in E(F)$,
then $a_{11}'y_6'\in E(F)$. Thus, $F-G_1=H_{108}^m$; see Fig. \ref{figB59}. If $a_{10}'y_2'\notin E(F)$, then there is a path $a_{10}'y_9'a_{12}'y_2'$. If $a_{12}'y_6'\in E(F)$, then
$a_{11}'y_9'\in E(F)$. Thus, $F-G_1=H_{109}^m$; see Fig. \ref{figB60}.  If $a_{12}'y_6'\notin E(F)$, then there is a path $a_{12}'y_{10}'a_{13}'y_6'$. Further, $a_{11}'$ and $a_{13}'$ have a common neighbor $y_{11}'$. Thus,  $F-G_1=H_{110}^m$ by Lemma \ref{lem1}; see Fig. \ref{figB61}. If $g_5^0$ is a hexagon, then there is a path $a_8'y_8'a_9'y_4'$.

 If $g_4^0$ is a quadrilateral, then $a_9'y_3'\in E(F)$ while $y_2'$ and $y_8'$ have a common neighbor $a_{10}'$. Thus,
 $H_{22}\subseteq F-G_1$ by Lemma \ref{lem1}; see Fig. \ref{figB62}. If $e_2\in E(g_4^0)$, then there is a path $a_9'a_{10}'y_3'$.
By the restriction on the faces of $F$,  $y_2'a_{10}'\notin E(F)$.
Hence, there is a path $y_2'a_{11}'y_9'a_{10}'$ while $y_8'$ and $y_9'$ have a common neighbor $a_{12}'$ by Observation \ref{claim2}(ii). If  $a_{11}'y_6'\in E(F)$, then $a_{12}'y_7'\in E(F)$. That is, $F-G_1=H_{111}^m$; see Fig. \ref{figB63}. If $a_{11}'y_6'\notin E(F)$, then there is a path $a_{11}'y_{10}'a_{13}'y_6'$. If $y_7'a_{13}'\in E(F)$, then $a_{12}'y_{10}'\in E(F)$. Thus, $F-G_1=H_{112}^m$; see Fig. \ref{figB64}. If $y_7'a_{13}'\notin E(F)$, then there is a path $y_7'a_{14}'y_{11}'a_{13}'$. Further, $a_{12}'$ and $a_{14}'$ have a common neighbor $y_{12}'$. By Lemma \ref{lem1}, $F-G_1=H_{113}^m$; see Fig. \ref{figB65}.
If $g_4^0$ is a hexagon, then there is a path $a_9'y_9'a_{10}'y_3'$.

 If $g_3^0$ is a quadrilateral, then $a_{10}'y_2'\in E(F)$ while $y_6'$ and $y_9'$ have a common neighbor $a_{11}'$. Thus,  $H_{22}\subseteq F-G_1$ by Lemma \ref{lem1}; see Fig. \ref{figB66}. If $e_2\in E(g_3^0)$, then there is a path $a_{10}'a_{11}'y_2'$.
By the restriction on the faces of $F$, $a_{11}'y_6'\notin E(F)$.
Hence, there is a path $a_{11}'y_{10}'a_{12}'y_6'$. Further, $y_9'$ and $y_{10}'$ have a common neighbor $a_{13}'$ by Observation \ref{claim2}(ii). If $a_{12}'y_7'\in E(F)$, then $a_{13}'y_8'\in E(F)$. Thus, $F-G_1=H_{114}^m$; see Fig. \ref{figB67}. If $a_{12}'y_7'\notin E(F)$, then there is a path $a_{12}'y_{11}'a_{14}'y_7'$.
 If $a_{14}'y_8'\in E(F)$, then $a_{13}'y_{11}'\in E(F)$. Thus, $F-G_1=H_{115}^m$; see Fig. \ref{figB68}. If $a_{14}'y_8'\notin E(F)$, then there is a path $a_{14}'y_{12}'a_{15}'y_8'$. Further, $a_{13}'$ and $a_{15}'$ have a common neighbor $y_{13}'$. Thus, $F-G_1=H_{116}^m$  by Lemma \ref{lem1}; see Fig. \ref{figB69}. If $g_3^0$ is a hexagon, then there is a path $a_{10}'y_{10}'a_{11}'y_2'$.

\begin{figure}[h]
    \centering

      \subfigure[$F-G_1=H_{115}^m$]{
    \label{figB68}
    \includegraphics[height=26mm]{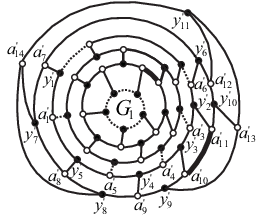}}
     \subfigure[$F-G_1=H_{116}^m$]{
    \label{figB69}
    \includegraphics[height=27mm]{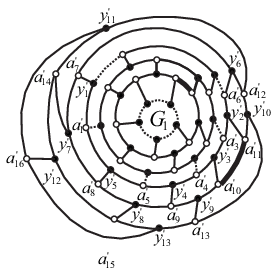}}
     \subfigure[$H_{22}\subseteq F-G_1$]{
    \label{figB70}
    \includegraphics[height=26mm]{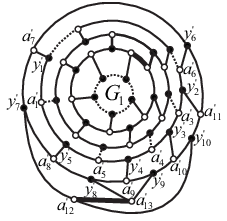}}
     \subfigure[$F-G_1=H_{117}^m$]{
    \label{figB71}
    \includegraphics[height=26mm]{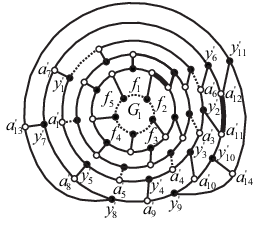}}
     \subfigure[$F-G_1=H_{118}^m$]{
    \label{figB72}
    \includegraphics[height=26mm]{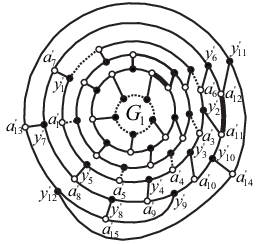}}
     \subfigure[$F-G_1=H_{119}^m$]{
    \label{figB73}
    \includegraphics[height=28mm]{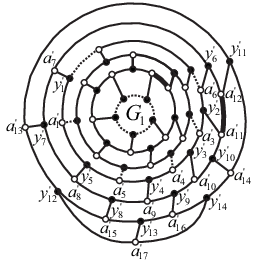}}
     \subfigure[$H_{22}\subseteq F-G_1$]{
    \label{figB74}
    \includegraphics[height=28mm]{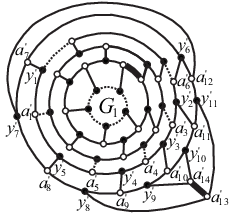}}
     \subfigure[$F-G_1=H_{120}^m$]{
    \label{figB76}
    \includegraphics[height=28mm]{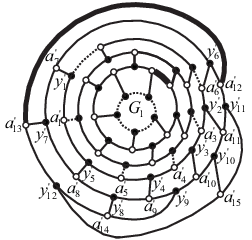}}
     \subfigure[$F-G_1=H_{121}^m$]{
    \label{figB77}
    \includegraphics[height=28mm]{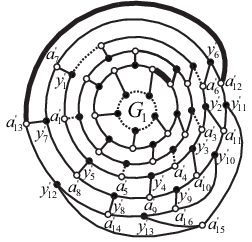}}
     \subfigure[$F-G_1=H_{122}^m$]{
    \label{figB78}
    \includegraphics[height=28mm]{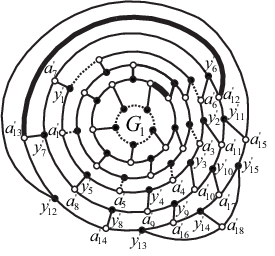}}
       \caption{Illustration for Subcase 2.4.2.}
\end{figure}

 If $g_2^0$ is a quadrilateral, then $a_{11}'y_6'\in E(F)$ while $y_{10}'$ and $y_7'$ have a common neighbor $a_{12}'$. Thus,  $H_{22}\subseteq F-G_1$  by Lemma \ref{lem1}; see Fig. \ref{figB70}. If $e_2\in E(g_2^0)$, then there is a path $a_{11}'a_{12}'y_6'$. Thus,
by the restriction on the faces of $F$, $a_{12}'y_7'\notin E(F)$.
Hence, there is a path $a_{12}'y_{11}'a_{13}'y_7'$. Further, $y_{10}'$ and $y_{11}'$ have a common neighbor $a_{14}'$. If $a_{13}'y_8'\in E(F)$, then $a_{14}'y_9'\in E(F)$. Thus, $F-G_1=H_{117}^m$; see Fig. \ref{figB71}. If $a_{13}'y_8'\notin E(F)$, then there is a path $a_{13}'y_{12}'a_{15}'y_8'$. If $a_{15}'y_9'\in E(F)$, then $y_{12}'a_{14}'\in E(F)$. Hence, $F-G_1=H_{118}^m$; see Fig. \ref{figB72}. If $a_{15}'y_9'\notin E(F)$, then there is a path $a_{15}'y_{13}'a_{16}'y_9'$. Further, $a_{16}'$ and $a_{14}'$ have a common neighbor $y_{14}'$. Thus, $F-G_1=H_{119}^m$  by Lemma \ref{lem1}; see Fig. \ref{figB73}.
If $g_2^0$ is a hexagon, then there is a path $a_{11}'y_{11}'a_{12}'y_6'$.

If $g_1^0$ is a quadrilateral, then $a_{12}'y_7'\in E(F)$ while $y_8'$ and $y_{11}'$ have a common neighbor $a_{13}'$. Thus, $H_{22}\subseteq F-G_1$ by Lemma \ref{lem1}; see Fig. \ref{figB74}. If $e_2\in E(g_1^0)$, then there is a path $a_{12}'a_{13}'y_7'$.
Further, 
by the restriction on the faces of $F$, $a_{13}'y_8'\notin E(F)$.
Thus, there is a path $a_{13}'y_{12}'a_{14}'y_8'$. Further, $y_{11}'$ and $y_{12}'$ have a common neighbor $a_{15}'$. If $a_{14}'y_9'\in E(F)$, then $a_{15}'y_{10}'\in E(F)$. Thus, $F-G_1=H_{120}^m$; see Fig. \ref{figB76}. If $a_{14}'y_9'\notin E(F)$, then there is a path $a_{14}'y_{13}'a_{16}'y_9'$. If $a_{16}'y_{10}'\in E(F)$, then $a_{15}'y_{13}'\in E(F)$. Thus, $F-G_1=H_{121}^m$; see Fig. \ref{figB77}.  If $a_{16}'y_{10}'\notin E(F)$, then there is a path $a_{16}'y_{14}'a_{17}'y_{10}'$. Further, $a_{15}'$ and $a_{17}'$ have a common neighbor $y_{15}'$. Thus,  $F-G_1=H_{122}^m$ by Lemma \ref{lem1}; see Fig. \ref{figB78}.


 \end{document}